\documentclass[10pt,reno]{amsart}
\usepackage{amssymb}
\usepackage{amsmath}
\usepackage{enumitem}
\usepackage{mathrsfs}
\usepackage[english]{babel}
\usepackage[all]{xy}
\usepackage{hyperref}

\usepackage{stmaryrd}

\usepackage[margin=1.05in]{geometry}

\RequirePackage{mathrsfs} 

\newtheorem{theorem}{Theorem}[section]
\newtheorem{lemma}[theorem]{Lemma}
\newtheorem{proposition}[theorem]{Proposition}
\newtheorem{corollary}[theorem]{Corollary}
\newtheorem{conjecture}[theorem]{Conjecture}

\theoremstyle{definition}
\newtheorem*{ack}{Acknowledgements}
\newtheorem*{con}{Conventions}
\newtheorem{remark}[theorem]{Remark}

\newtheorem{definition}[theorem]{Definition}

\numberwithin{equation}{section} \numberwithin{figure}{section}

\DeclareMathOperator{\Spec}{Spec}

\DeclareMathOperator{\an}{an}

\DeclareMathOperator{\Hom}{Hom}

\DeclareMathOperator{\ord}{ord}

\newcommand\QQ{\mathbb{Q}}

\newcommand\CC{\mathbb{C}}

\newcommand{\et}{\textrm{\'{e}t}}

\usepackage{color}

\definecolor{orange}{rgb}{1,0.5,0}

\title[Albanese maps and fundamental groups of geometrically-special varieties]{Albanese maps and fundamental groups of varieties with many rational points over function fields}

\author{Ariyan Javanpeykar}
\address{Ariyan Javanpeykar \\
	Institut f\"{u}r Mathematik\\
	Johannes Gutenberg-Universit\"{a}t Mainz\\
	Staudingerweg 9, 55099 Mainz\\
	Germany.}
\email{peykar@uni-mainz.de}

\author{Erwan Rousseau}
\address{Erwan Rousseau \\ Institut Universitaire de France
	\& Aix Marseille Univ\\
		CNRS, Centrale Marseille, I2M\\
		Marseille\\
		France}
\email{erwan.rousseau@univ-amu.fr }

\thanks{The second author was supported by the ANR project \lq\lq FOLIAGE\rq\rq{}, ANR-16-CE40-0008  and by the European Union’s Horizon 2020 research and innovation program under the Marie Sklodowska-Curie grant agreement No 75434}

 
\begin{document}

\begin{abstract}  
We investigate properties of the Albanese map and the fundamental group of a complex projective variety with  many rational points over some function field, and prove that every linear quotient of the fundamental group of such a variety  is virtually abelian, as well as  that its Albanese map   is surjective, has connected fibres, and has no      multiple fibres in codimension one. 
\end{abstract}

\maketitle

\section{Introduction}

The abundance of rational points on a variety over a number field is conjectured to force the   fundamental group of the underlying topological space to exhibit various rigidity properties. Indeed, it seems reasonable to suspect that the topological fundamental group of   a smooth projective variety  over a number  field with a dense set of rational points contains a finite index abelian subgroup.  In this paper we verify   such expectations for varieties with a dense set of rational points over a function field of characteristic zero.
 
  \subsection{Lang's conjectures}
Let $k$ be an algebraically closed field of characteristic zero.  
  Lang conjectured that varieties of general type over $k$ satisfy certain finiteness properties (see \cite{Lang2} and \cite[\S12]{JBook}), where  we say that a proper (integral) variety $X$ over $k$ is of general type if for some (hence any) proper desingularization $\widetilde{X}\to X$ we have that $\omega_{\widetilde{X}}$ is big.  To state the relevant version of Lang's conjecture, we follow \cite{vBJK, JBook} and introduce the class of pseudo-geometrically hyperbolic varieties.

\begin{definition}[Pseudo-geometric hyperbolicity]\label{def:psgeomhyp}  A variety  $X$ over $k$ is   \emph{pseudo-geometrically hyperbolic over $k$} if there is a proper closed subset $Z\subsetneq X$ such that,   for every $x$ in $X(k)\setminus Z$ and every smooth connected pointed curve $(C,c)$ over $k$,  the set $\Hom_k((C,c),(X,x))$ of morphisms $f:(C,c)\to (X,x)$  is finite.  
\end{definition}

 We will also say that a variety $X $ over $k$ is \emph{geometrically hyperbolic over $k$} if one can take $Z=\emptyset$ in the above definition i.e., for every $x$ in $X(k)$ and every smooth connected pointed curve $(C,c)$ over $k$, the set $\Hom_k((C,c),(X,x))$ is finite. Note that ``being geometrically hyperbolic over $k$'' means, roughly speaking, that, for any smooth connected curve $C$ with function field $K$,   the variety $X_{k(t)}$ over the function field $k(t)$ has only finitely many  ``pointed'' $K$-rational points. Similarly, ``being pseudo-geometrically hyperbolic'' is similar to having such a finiteness property for ``pointed'' rational points \emph{outside} some   exceptional locus $Z$.

 The starting point of our paper is the following finiteness conjecture for varieties.  
 
\begin{conjecture}[Consequence of Lang's conjectures]\label{conj:lang}
A projective (integral) variety $X$ over $k$ is of general type if and only if it is pseudo-geometrically hyperbolic over $k$.
\end{conjecture}

For the reader's convenience we briefly explain how Conjecture \ref{conj:lang} is related to Lang's conjectures, as this might be unclear to non-experts. First, Lang's conjectures   predict that     a projective   variety of general type is pseudo-algebraically hyperbolic (Definition \ref{def:alghyp2}), and a simple application of Mori's bend-and-break shows that a pseudo-algebraically hyperbolic scheme is pseudo-geometrically hyperbolic (see Proposition \ref{prop:alg_hyp_implies_geom_hyp}).
On the other hand, a pseudo-algebraically hyperbolic variety is pseudo-groupless \cite[Proposition~5.4]{JXie}, and Lang conjectured that pseudo-groupless projective varieties    are of general type. In fact, the latter conjecture also follows from combining the abundance conjecture in the MMP  with ``standard'' conjectures on rational curves on Calabi-Yau varieties. For example, note    that groupless \emph{surfaces} are of general type   (see  \cite[Lemma~3.23]{JXie}).
 
The above consequence of Lang's conjecture (which we simply refer to as Lang's conjecture) is known in dimension one, but open in dimension two. We do note that   every proper pseudo-geometrically hyperbolic surface is of general type (see \cite[Proposition~11.7]{JBook} for example);  such a statement is not known currently in dimension  three. Moreover, as we show in Section \ref{section:connectedness}, Lang's conjecture holds for normal projective varieties  whose Albanese map is generically finite onto its image. Finally, for other non-trivial examples for which we know Lang's conjecture to hold we refer the reader to Section \ref{section:psgeomhyp}.
   
 Lang's conjecture is concerned with the \emph{finiteness} of rational points on varieties over function fields. We will however be interested in varieties with a lot of rational points over function fields, i.e., those that behave ``opposite'' to pseudo-geometrically hyperbolic varieties. This brings us to Campana's conjectures on special varieties.
 
 \subsection{Campana's conjectures}
Campana's conjectures provide a picture    complementary to that of Lang's conjectures on rational points and entire curves on varieties of general type.  To explain this,  for $X$  a   proper   variety over $k$,   we follow Campana and say that $X$ is \emph{special (over $k$)}  if, for some (hence every) desingularization  $\widetilde{X}\to X$, for every integer $p\geq 1$ and every rank one coherent subsheaf $F\subset \Lambda^p \Omega^1_{\widetilde{X}}$,   the Kodaira dimension of $F$ is at most $p-1$; see   \cite{Ca04, Ca11}. Roughly speaking, a variety is special if it is as far away as possible from being of general type. Just like being of general type, this notion is insensitive to field extensions, i.e., if $L/k$ is an extension of algebraically closed fields of characteristic zero and $X$ is a proper variety over $k$, then $X$ is special over $k$ if and only if $X_L$ is special over $L$ (see \cite[Proposition~9.18]{Ca04}).

Before we state the version of Campana's conjecture   we are interested in, we recall his analytic and arithmetic conjectures. 

\begin{definition}
A variety $X$ over $\mathbb{C}$ is \emph{Brody-special} if there is a  (Zariski-)dense holomorphic map $\mathbb{C}\to X^{\an}$.
\end{definition}

\begin{conjecture}[Campana's analytic conjecture]\label{conj:campana_an}
A   projective   variety $X$ over $\mathbb{C}$ is special   if and only if it is Brody-special.
\end{conjecture}

A non-constant holomorphic map $\mathbb{C}\to X^{\an}$ is usually referred to as an entire curve in $X$, and a variety $X$ over $\mathbb{C}$ is said to be \emph{pseudo-Brody hyperbolic} if there is a proper (Zariski-)closed subset $Z\subsetneq X$ such that every entire curve $\mathbb{C}\to X^{\an}$ factors over $Z^{\an}$.
Just as   positive-dimensional special varieties are as far away as possible from being of general type,   a positive-dimensional Brody-special variety is as far away as possible from being pseudo-Brody hyperbolic.   Therefore,    Campana's analytic conjecture provides a complementary picture to the Green--Griffiths--Lang conjecture that a proper   variety is of general type if and only if it is pseudo-Brody hyperbolic.  (To be precise, Green--Griffiths and Lang explicitly   conjectured that a projective variety of general type is pseudo-Brody hyperbolic. The converse statement is implicit in Lang's paper; see \cite[p.~161]{Lang2}.)

 Intuitively speaking, the existence of a  dense entire curve  on a positive-dimensional variety over a number field should be equivalent to the existence of many rational points. To make this   more precise,  we extend Campana's notion of arithmetically-special variety to varieties over arbitrary algebraically closed fields of characteristic zero.
 
 \begin{definition}  A proper variety $X$ over   $k$  is \emph{arithmetically-special over $k$} if there exists a  subfield $K\subset k$ which is a finitely generated field extension of $\mathbb{Q}$ and a model $\mathcal{X}$ for $X$ over $K$ such that  $\mathcal{X}(K)$ is dense in $X$.
 \end{definition}
 
 The notion of ``arithmetic-specialness'' introduced here is slightly more general than that of Campana who restricts himself to varieties over $\overline{\mathbb{Q}}$. (Note that a variety over $\overline{\mathbb{Q}}$ is arithmetically-special  in the above sense if and only if there is a number field $K\subset \overline{\mathbb{Q}}$ and a model $\mathcal{X}$ for $X$ over $K$ such that $\mathcal{X}(K)$ is dense.)

The arithmetic analogue of Conjecture \ref{conj:campana_an}   then reads as follows.

\begin{conjecture}[Campana's arithmetic conjecture]\label{conj:campana_ar} A  projective   variety $X$  over $k$  is special over $k$ if and only if it is arithmetically-special over $k$. 
\end{conjecture}

A proper variety $X$ over $k$ is  \emph{pseudo-Mordellic over $k$} if  there is a proper closed subset $Z\subsetneq k$ such that,  for every   subfield $K\subset k$ which is a finitely generated field extension of  $\mathbb{Q}$ and every model $\mathcal{X}$ for $X$ over $K$, the set $\mathcal{X}(K)\setminus Z$ is finite.  Just as   positive-dimensional special varieties are as far away as possible from being of general type,  a positive-dimensional   arithmetically-special variety is as far away as possible  from being pseudo-Mordellic.  Therefore, Campana's conjecture complements Lang's conjecture that a projective   variety $X$ over $k$ is of general type if and only if $X$ is pseudo-Mordellic over $k$.

Inspired by Campana's analytic and arithmetic conjectures, and led by the analogy between function fields and number fields (see for instance \cite[Remark~11.3]{JBook}), we introduce a ``special'' counterpart to the notion of pseudo-geometric hyperbolicity (Definition \ref{def:psgeomhyp}).

\begin{definition} [Geometrically special varieties] \label{def:geom} A variety $X$ over $k$ is \emph{geometrically-special over $k$} if, for every dense open subset $U\subset X$, there exists a smooth quasi-projective connected curve $C$ over $k$, a point $c$ in $C(k)$, a point $u$ in $U(k)$, and a sequence of  morphisms $f_i:C\to X$   with $f_i(c) = u$ for $i=1,2,\ldots$ such that $C\times X $ is covered by the graphs $\Gamma_{f_i}\subset C\times X$ of these maps, i.e., the closure of $\cup_{i=1}^{\infty} \Gamma_{f_i}$ equals $ C\times X$. 
\end{definition}

Note that the density of $\cup_{i=1}^\infty \Gamma_{f_i}$ in this definition is equivalent to  the density of the $K(C)$-rational points of $X_{K(C)}$ induced by     morphisms   $C\to C\times X$ with $f(c) = (c,x)$. That is, roughly speaking, a variety is geometrically-special if and only if the set of ``pointed'' rational points of $X_{k(t)}$ is potentially dense.  Thus, intuitively speaking, to be geometrically-special is to admit an abundant set of rational points over some function field.

 Just as positive-dimensional special varieties are   as far away from being of general type, a positive-dimensional geometrically-special is as far away as possible from being pseudo-geometrically hyperbolic. For example,  a  geometrically-special variety does not dominate a positive-dimensional pseudo-geometrically hyperbolic variety (see Corollary \ref{cor:image_is_not_psgeom}).
 
Following Campana's philosophy on special varieties, we expect that  geometric-specialness coincides with being special.

\begin{conjecture}[Inspired by Campana]\label{conj:campana_geom} A   projective    variety $X$  over $k$  is special over $k$ if and only if it is geometrically-special over $k$. 
\end{conjecture}

The aim of this paper is to provide evidence for Conjecture \ref{conj:campana_geom}. We stress that this conjecture fits in well with the parellelism between Campana's and Lang's conjecture, as illustrated by the following table.

\begin{center}
\begin{tabular}{ |p{6cm}|p{6cm}|  }
\hline
\multicolumn{2}{|c|}{Lang and Campana conjectures} \\
\hline
Notion of pseudo-hyperbolicity      & Notion of specialness   \\
\hline
\\[-1em]
  General type &  Special\\
 Pseudo-Brody-hyperbolic  &     Brody-special \\
Pseudo-Mordellic &   Arithmetically-special \\ 
Pseudo-geometrically-hyperbolic  & Geometrically-special.  \\
\hline
\end{tabular}
 \end{center}

 Lang's conjecture says that the four notions on the left hand side are   equivalent for any projective variety $X$ over $k$, whereas Campana's conjecture and Conjecture \ref{conj:campana_geom} predict that the four notions on the right hand side are equivalent. 
 Except for some particular cases (e.g,. $X$ a closed subvariety of an abelian variety) both conjectures are still wide open.

\subsection{The Albanese map of a special variety} 
 We will use the following terminology.
  \begin{definition}\label{defn:mult}
A surjective morphism $f:X\to Y$ of   varieties over $k$ is said to have \emph{no multiple fibres in codimension one} if, for every point $y$ in $Y$ of codimension one, the scheme-theoretic fibre $X_y$ has an irreducible component which is generically reduced.       More concretely, if $f:X\to Y$ has no multiple fibres in codimension one and  we write $X_y = \sum_F n_F F$, where the sum runs over all irreducible components of $X_y$ and $n_F$ is the multiplicity of $F$ in $X_y$ (see \cite[Tag~02QU]{stacks-project}), then  there is a component $F$ of $X_y$ such that $n_F=1$. 
\end{definition}

Our first main result concerns the structure of the Albanese map of a geometrically-special normal projective variety. We recall that, given a normal proper (integral) variety over $k$, the Albanese variety of $X$ is $\mathrm{Alb}(X):=(\mathrm{Pic}^0_{X/k})_{\mathrm{red}}$, and  we refer the reader to \cite[Proposition~A.6]{Mochizu} for basic properties of the Albanese map. (Strictly speaking, the Albanese map $X\to \mathrm{Alb}(X)$ is associated to the choice of a point $x$ in $X(k)$. However, the properties we are interested in are independent of the choice of $x$ in $X(k)$.)
 
 \begin{theorem}[Main Result, I]\label{thm:main_result_I}
   If $X$ is  a normal projective \emph{geometrically-special} variety over $k$, then the Albanese map $X\to \mathrm{Alb}(X)$ is surjective, has connected fibres, and has no multiple fibers in codimension one.
 \end{theorem}
 
Our proof of Theorem \ref{thm:main_result_I} relies   on  Yamanoi's seminal work on holomorphic curves in algebraic varieties of maximal Albanese dimension \cite{Yamanoi1}. In particular, although the statements are purely algebraic, we implicitly rely on Yamanoi's work on (complex-analytic) Nevanlinna theory.

Theorem \ref{thm:main_result_I} fits in perfectly with Conjecture \ref{conj:campana_geom} as the following result of Campana (see \cite[Proposition~5.3]{Ca04}) shows.  

\begin{theorem}[Campana] \label{thm:alb_of_special}
If $X$ is  a normal projective \emph{special} variety over  $k$, then the Albanese map $X\to \mathrm{Alb}(X)$ is surjective, has connected fibres, and has no multiple fibres in codimension one.
\end{theorem}

We mention that, as explained in   \cite[Proposition~5.1]{CampanaClaudonStability}, 
for $X$   a smooth projective \emph{Brody-special} variety over  $\CC$,   the Albanese map $X\to \mathrm{Alb}(X)$ is surjective, has connected fibres, and has no multiple fibres in codimension one (see also Remark \ref{remark:analytic}).  
On the other hand, 
for $X$ a smooth projective \emph{arithmetically-special} variety over $k$, the Albanese map $X\to \mathrm{Alb}(X)$   is surjective by Faltings's theorem \cite{Fa94}. In light of Campana's conjectures, the fibers of the Albanese map should even be connected. However, as we do not have at our disposal  any satisfying results on the finiteness of rational points for ramified covers of abelian varieties (see \cite{CDJLZ} for recent progress nonetheless), proving this connectedness seems out of reach at the moment.

\subsection{The abelianity conjecture}

For $X$ a connected variety over $\CC$, we let $\pi_1(X)$ be the (topological) fundamental group of $X^{\an}$ (with respect to the choice of some basepoint). Note that the \'etale fundamental group $\pi_1^{\et}(X)$ of $X$ is the profinite completion of $\pi_1(X)$.

Campana conjectured that the fundamental group $\pi_1(X)$ of a smooth projective connected special variety over $\CC$ is virtually abelian (i.e., contains a finite index abelian subgroup).   
\begin{conjecture}[Campana's abelianity conjecture]\label{conj:campana2} Let $X$ be a smooth projective connected variety over $\CC$.
Then the following statements hold. 
\begin{enumerate}
\item If $X$ is special, then $\pi_1(X)$ is virtually abelian.
\item If $X$ is Brody-special, then $\pi_1(X)$ is virtually abelian.
\item If $X$ is arithmetically-special over $\mathbb{C}$, then $\pi_1(X)$ is virtually abelian.
\item If $X$ is geometrically-special over $\mathbb{C}$, then $\pi_1(X)$ is virtually abelian.
\end{enumerate}
\end{conjecture}

This conjecture is still open. However, under the additional assumption that $\pi_1(X)$ admits a faithful linear complex representation, this conjecture is resolved in cases $(1)$ and $(2)$ by Campana and Yamanoi, respectively. Their precise results  are as follows (see  \cite[Theorem~7.8]{Ca04} and \cite[Theorem~1.1]{Yam10}).

\begin{theorem}[Campana]\label{thm:pi1_campana}
 If $X$ is a smooth projective \emph{special} variety over $\CC$, then the image of any homomorphism $\pi_1(X)\to \mathrm{GL}_n(\CC)$ is virtually abelian. 
\end{theorem}

\begin{theorem}[Yamanoi]\label{thm:pi1_yamanoi}
 If $X$ is a smooth projective \emph{ Brody-special} variety over $\CC$, then the image of any homomorphism $\pi_1(X)\to \mathrm{GL}_n(\CC)$ is virtually abelian. 
\end{theorem}

Motivated by Conjecture \ref{conj:campana_geom}, Conjecture  \ref{conj:campana2}, Campana's theorem (Theorem \ref{thm:pi1_campana}), and Yamanoi's theorem (Theorem \ref{thm:pi1_yamanoi}), we  prove the following result.

\begin{theorem}[Main Result, II] \label{thm:pi1_geomhyp}
 If $X$ is a smooth projective \emph{geometrically-special} variety over $\CC$, then the image of any homomorphism $\pi_1(X)\to \mathrm{GL}_n(\CC)$ is virtually abelian.    
\end{theorem}

Our final result pushes Theorem \ref{thm:pi1_geomhyp}   a bit further. To state it, we follow \cite[p.~552]{Yam10} and recall that a representation $\rho:\pi_1(X)\to \mathrm{GL}_n(\CC)$ is   \emph{big} if there is a countable  collection of proper closed subsets $X_i\subsetneq X$ with $i=1,\ldots$ such that, for every positive-dimensional subvariety $Z\subset X$ containing a point of $X\setminus \cup_{i=1}^\infty X_i$, the group $\rho[\mathrm{Im}(\pi_1(\tilde{Z})\to \pi_1(X))]$ is infinite, where   $\tilde{Z}\to Z$ is a desingularization of $Z$.
 
 \begin{theorem}[Main Result, III]\label{thm:pi1_geomhyp2}
 Let $X$ be a smooth projective connected variety over $\CC$ with a big representation $\rho:\pi_1(X)\to \mathrm{GL}_n(\CC)$. If $X$ is geometrically-special, then there exists a finite \'etale covering $Y\to X$ which is birationally equivalent to an abelian variety.
 \end{theorem}
 
 The proofs of Theorem \ref{thm:pi1_geomhyp} and \ref{thm:pi1_geomhyp2}  rely on  our structure result for the Albanese map of a geometrically-special variety (Theorem \ref{thm:main_result_I}),   foundational results from Hodge theory due to Deligne, Griffiths, and Schmid (see Lemma \ref{lem:deligne}),  as well as the recent resolution of Griffiths's algebraicity conjectures for period maps by Bakker-Brunebarbe-Tsimerman (see Theorem \ref{thm:period_maps}),   Zuo's results on spectral covers (see the proof of Theorem \ref{prop:3.1a}),  and the theory of Shafarevich maps (see the proof of Proposition \ref{prop:3.1}). Arguably, the novel technical result in our proof of Theorem \ref{thm:pi1_geomhyp} is  Theorem \ref{prop:3.1a}.

 \subsection{Outline of paper} 
 In Section \ref{section:geom_special} we prove basic properties of geometrically-special varieties such as the invariance by finite \'etale covers (Lemma \ref{lem:cw_geom_special}), as well as the fact that the image of a geometrically-special variety is geometrically-special (Lemma \ref{lem:surj}). In Section \ref{section:abelian varieties} we show that abelian varieties are geometrically-special (Proposition \ref{prop:ab_var_are_special}), and prove the structure result for the Albanese map of a geometrically-special projective variety (Theorem \ref{thm:main_result_I}) by applying Yamanoi's results on the pseudo-hyperbolicity of ramified covers of abelian varieties, as well as Yamanoi's work on the pseudo-hyperbolicity of orbifold pairs on abelian varieties (see Section \ref{section:yamanoi_orb}). In Section \ref{section:pi1}, we prove that linear quotients of the fundamental group of a smooth projective connected geometrically-special variety  are virtually abelian.  
 
 \begin{ack}
The ideas of Campana and Yamanoi  have had an enormous impact on this paper. We are grateful to both of them for their past work without which this paper would not have existed.  We thank Philipp Habegger and Junyi Xie for telling us about Lemma \ref{lem:countable}. The first named author thanks the IHES and  the University of Paris-Saclay for their hospitality. Part of this work was done during the visit of the first and second named author at the Freiburg Institute for Advanced Studies; they thank the Institute for providing an excellent working environment.
We are grateful to the referees for many useful comments.
 \end{ack}
 
 \begin{con} Throughout this paper, we let $k$ be an algebraically closed field of characteristic zero.  A \emph{variety} over $k$ is a finite type separated integral (i.e., irreducible and reduced) scheme over $k$.

 If $X$ is a variety over $k$  and   $A\subset k$ is a subring, then a \emph{model for $X$ over $A$} is a pair $(\mathcal{X},\phi)$ with $\mathcal{X}\to \Spec A$ a finite type separated scheme and $\phi:\mathcal{X}\otimes_A k \to X$ an isomorphism of schemes over $k$. We  omit $\phi$  from our notation.
 
 A \emph{pointed variety (over $k$)}  is a pair $(V,v)$, where  $V$ is a variety over $k$ and $v$ is an element of $V(k)$. A pointed variety $(C,c)$ is a \emph{pointed curve} if $C$ is pure of dimension one. We say that a pointed variety $(V,v)$ is smooth (resp. connected, resp. projective) if $V$ is smooth (resp. connected, resp. projective). A morphism $(V,v)\to (V',v')$ of pointed varieties consists of the data of a morphism $f:V\to V'$ of varieties over $k$ such that $f(v) = v'$.

 If $X$ is a locally finite type scheme over $\mathbb{C}$, we let $X^{\an}$ be the associated complex-analytic space.  
 \end{con}
 
\section{Geometrically-special varieties}\label{section:geom_special}
Guided by the expectation that geometric-specialness coincides with Campana's notion of specialness (Conjecture \ref{conj:campana_geom}),  
we will prove several basic properties of geometrically-special varieties in this section; none of the results in this section are particularly hard to prove.

Let $k$ be an algebraically closed field of characteristic zero.
Recall that   a variety $X$ over $k$ is  {geometrically-special over $k$} if, for every dense open subset $U\subset X$, there exists a smooth affine connected pointed curve $(C,c)$, a point $x$ in $U(k)$, and a sequence of   morphisms $\{f_i:(C,c)\to (X,u)\}_{i=1}^\infty$  such that  the closure of $\cup_{i=1}^{\infty} \Gamma_{f_i}$  in $C\times X$ equals $ C\times X$ (see Definition \ref{def:geom}). 
We will sometimes refer to a sequence of morphisms $f_i:(C,c)\to (X,u)$ such that  the closure of $\cup_{i=1}^\infty \Gamma_{f_i}$ equals $C\times X$ as a \emph{covering set} for $X$ (even though it is really a ``covering'' of $C\times X$).

Clearly, a proper variety $X$ over $k$ is geometrically-special over $k$ if and only if, for every dense open subset $U\subset X$, there exists a smooth  \emph{projective} connected pointed curve $(C,c)$ over $k$,  a point $u$ in $U(k)$, and a covering set $f_i:(C,c)\to (X,u)$ for $X$. The existence of a covering set here is in fact equivalent to the   the universal evaluation map 
\[
C\times \underline{\Hom}_k((C,c),(X,x))\to C \times X
\] being dominant, where    $\underline{\Hom}_k((C,c),(X,x))$ denotes the moduli scheme of morphisms $(C,c)\to (X,x)$ over $k$; this moduli scheme is a countable union of quasi-projective varieties over $k$.

Being geometrically-special is related to   potential density of ``pointed'' rational points, as we briefly explain    in the following remark.

\begin{remark} \label{remark:potential_density} Let $K$ be the function field of a smooth connected curve $C$ over $k$, and let $X$ be a variety over $\overline{K}$. We say that $X$ \emph{satisfies $k$-pointed-potential density} if, for every dense open subset $U \subset X$, there is a point $u$ in $U$, a curve $D$ over $k$ quasi-finite over $C$, a point $d$ in $D(k)$, and a model $(\mathcal{X},x)$ for $(X,x)$ over $D$  such that 
\[
\mathrm{Im}[ \Hom_C( (D,d), (\mathcal{X},x ) )  \to X(\overline{K})]
\] is dense in $X$. This extends the above definition in the following sense. A variety $X$ over $k$ is geometrically-special over $k$ if and only if (the ``constant'' variety) $X_{\overline{k(t)}}$ satisfies $k$-pointed-potential density.  
\end{remark}

   Note that Conjecture \ref{conj:campana_geom} is concerned with isotrivial varieties over $k(t)$.  On the other hand, even in the non-isotrivial case, it seems reasonable to supect density of   ``pointed rational points''; a similar conjecture was stated for (non-pointed) rational points on non-isotrivial varieties by Campana \cite{Ca11}.

   \subsection{Algebraic and geometric hyperbolicity}\label{section:psgeomhyp}
 To keep track of the ``exceptional locus'' of   a pseudo-geometrically hyperbolic variety (Definition \ref{def:psgeomhyp2}), we   use the notion of geometric hyperbolicity modulo a closed subset. 
 
\begin{definition} \label{def:psgeomhyp2}   Let $X$ be a variety over $k$. If $Z$ is a closed subset of $X$, then $X$ is \emph{geometrically hyperbolic modulo $Z$} if,    for every $x$ in $X(k)\setminus Z$ and every smooth connected pointed curve $(C,c)$ over $k$,  the set $\Hom_k((C,c),(X,x))$    is finite.
\end{definition}

Bounding the degrees of maps from curves into a compact space   forces finiteness properties of sets of pointed maps (see \cite{JKa} for precise statements). To make this more precise, we recall Demailly's notion of algebraic hyperbolicity \cite{Dem97}; we discuss this notion further in the orbifold setting in Section \ref{section:psalg}.

\begin{definition} \label{def:alghyp1}Let $ X$ be a projective variety and let $E\subset X$ be a closed subset. Then    $X$  is \emph{algebraically hyperbolic modulo $E$ over $k$} if, for every ample line bundle $L$ on $X$, there is a constant $\alpha_{X,L,E}$ such that, for every smooth projective connected curve $C$ over $k$, and every   morphism $f:C\to X$ with $f(C)\not\subset  E$, the inequality
\[
\deg_C f^\ast L \leq \alpha_{X,L,E} \cdot \mathrm{genus}(C)
\] holds.
\end{definition}

\begin{definition}\label{def:alghyp2}
A projective variety $X$ is \emph{pseudo-algebraically hyperbolic over $k$} if there is a proper closed subset $E\subsetneq X$   such that $X$ is algebraically  hyperbolic modulo $E$.
\end{definition}

We will use the following simple application of Mori's bend-and-break (see   \cite[Theorem~1.12.(2)]{JXie}).
\begin{proposition}
\label{prop:alg_hyp_implies_geom_hyp}
If $X$ is a projective pseudo-algebraically hyperbolic scheme over $k$, then $X$ is pseudo-geometrically hyperbolic over $k$. \qed
\end{proposition}

Since a positive-dimensional projective geometrically-special scheme over $k$ is not pseudo-geometrically hyperbolic, it follows from Proposition \ref{prop:alg_hyp_implies_geom_hyp} that such a variety is also not pseudo-algebraically hyperbolic.   

For the reader's convenience, we give several examples of  pseudo-geometrically hyperbolic varieties. 
For example,  if $X$ is a proper Brody hyperbolic variety over $\CC$, then   $X$ is geometrically hyperbolic over $\CC$ (see      \cite[Theorem~5.3.10]{KobayashiBook}), and even   algebraically hyperbolic \cite{Dem97}. In the non-compact case, if $X$ is a smooth affine variety which admits a hyperbolic embedding \cite[Chapter~3.3]{KobayashiBook}, then $X$ is geometrically hyperbolic (see \cite[Theorem~1.7]{JLevin}). 
 Furthermore, if $X$ is a smooth projective connected surface of general type with $c_1^2>c_2$,  
then  the projective surface $X$ is pseudo-algebraically hyperbolic  (and thus pseudo-geometrically hyperbolic) by Bogomolov's theorem \cite{Bogomolov}.
Finally, examples of geometrically hyperbolic moduli spaces of polarized varieties are given in \cite{JLitt, JSZ}.

 \subsection{Birational invariance}
 The notions of being special, Brody-special, and arithmetically-special are ``birational''.  We now show  that the same holds for being geometrically-special (see Lemma \ref{lem:surj}).
\begin{lemma}[Going up]\label{lem:surj-1}
Let $X\to Y$ be a proper birational morphism of   varieties over $k$. If $Y$ is geometrically-special over $k$, then $X$ is geometrically-special.
\end{lemma}
\begin{proof} 
Let $y\in Y$ be  a point not contained in the exceptional locus of $X\to Y$ for which we can choose a pointed curve $(C,c)$ and a covering set  $f_i:(C,c)\to (Y,y)$.  Let $x\in X$ be the unique point of $X$ lying over $y$. Then, as $y$ is not contained in the exceptional locus of the proper morphism $X\to Y$, we have that each $f_i$ lifts uniquely to a map $g_i:(C,c)\to (X,x)$. It is clear that  these maps form a covering set for $(X,x)$, i.e., the union of the graphs $\Gamma_{g_i}$ in $C\times X$ is dense in $X$. Clearly, this implies that $X$ is geometrically-special over $k$. 
\end{proof}

 \begin{lemma}[Going down]\label{lem:surj0}
 Let $f:X\to Y$ be a surjective morphism of  varieties over $k$.  If $X$ is geometrically-special over $k$, then $Y$ is geometrically-special over $k$.
 \end{lemma}
 \begin{proof}
 Let $V\subset Y$ be a dense open subset, and let $U:=f^{-1}(V)$. Since $X$ is geometrically-special over $k$, there is a pointed curve $(C,c)$ over $k$, a point $u$ in $U(k)$, and a covering set   $f_i:(C,c)\to (X,x)$. Define $v:= f(u)$, and  $g_i:= f\circ f_i$. Then $g_i(c) = v$ and $C\times Y$ is covered by the graphs of the maps $g_i$, so that the $g_i:(C,c)\to (Y,v)$ form a covering set for $Y$.  We conclude that $Y$ is geometrically-special over $k$, as required.
 \end{proof}
 
\begin{lemma}\label{lem:surj} Let $X\dashrightarrow Y$ be a dominant rational  map of proper   varieties over $k$. Then the following statements hold.
\begin{enumerate}
\item If $X$ is special, then $Y$ is special.
\item Assume $k=\CC$. If $X$ is Brody-special, then $Y$  is Brody-special.
\item If $X$ is arithmetically-special over $k$, then $Y$ is arithmetically-special over $k$.
\item If $X$ is geometrically-special over $k$, then $Y$ is geometrically-special over $k$.
\end{enumerate} 
\end{lemma}
\begin{proof} Note that $(1)$ is due to Campana \cite{Ca04}, and that $(2)$ and $(3)$ are obvious. To prove $(4)$, let $Z\to X$ be a proper birational surjective morphism such that the composed rational map $Z\to X\dashrightarrow Y$ is a (proper surjective) morphism. First, as $X$ is geometrically-special over $k$, it follows that $Z$ is geometrically-special (Lemma \ref{lem:surj-1}). Then, as $Z$ is geometrically-special over $k$ and $Z\to Y$ is surjective, we conclude that $Y$ is geometrically-special over $k$ (Lemma \ref{lem:surj0}).
\end{proof}

\begin{corollary}\label{cor:image_is_not_psgeom}
 Let $X\dashrightarrow Y$ be a dominant rational map of proper   varieties. If $\dim Y\geq 1$ and $X$ is geometrically-special over $k$, then $Y$ is not pseudo-geometrically hyperbolic over $k$.  
\end{corollary}
\begin{proof}
By Lemma \ref{lem:surj}.(4) we have that $Y$ is geometrically-special. This proves the corollary, as a positive-dimensional geometrically-special variety is  not pseudo-geometrically hyperbolic.
\end{proof}

\subsection{Finite \'etale maps}  The fact that specialness is preserved under taking finite \'etale covers was proven by Campana. The corresponding result for geometrically-special varieties requires  a well-known finiteness result for finite \'etale covers of varieties. Note that, in this paper,  the following   finiteness result is only used in the case that $X$ is a curve (in which case it is easier to prove); we include the more general statement for the sake of future reference. 

\begin{lemma}[Finiteness]\label{lem:finiteness}
Let $d\geq 1$ be an integer, and let $X$ be a variety over $k$. Then, the set of $X$-isomorphism classes of finite \'etale morphisms $Y\to X$ of degree $d$ is finite.
\end{lemma} 
\begin{proof} We may and do assume that  $k=\CC$. Then, the result follows from the fact that  the (topological) fundamental group $\pi_1(X)$ of $X$ is finitely generated (see \cite[Expos\'e~II, Th\'eor\`eme~2.3.1]{SGA7I}).
\end{proof}

\begin{lemma}\label{lem:cw_geom_special} Let $\pi\colon X\to Y$ be a finite \'etale morphism of   varieties over $k$. Then $X$ is geometrically-special over $k$ if and only if $Y$ is geometrically-special over $k$.
\end{lemma}
\begin{proof}  Assume that $X$ is geometrically-special over $k$.
Since $\pi:X\to Y$ is surjective, it follows from  Lemma \ref{lem:surj} that $Y$ is geometrically-special over $k$.

 Conversely, assume that $Y$ is geometrically-special over $k$.  Let $U\subset X$ be a dense open subset such that $\pi|_U:U\to \pi(U)$ is finite \'etale of degree $\deg \pi$. Let $y$ be a point in  the open subset $\pi(U)$ of $Y$ for which we can choose a covering set $f_i:(C,c)\to (Y,y)$ with $i=1,2,\ldots $.
For each $i$, consider the Cartesian diagram
\[
\xymatrix{ D_i  \ar[d]_{p_i} \ar[rr] & & X \ar[d]^{\pi} \\ C \ar[rr]_{f_i} & & Y}
\]
Since $p_i$ is a finite \'etale (surjective) morphism of degree at most $\deg(\pi)$, the set of  isomorphism classes of the curves $D_i$ is finite (Lemma \ref{lem:finiteness}). Thus, replacing the $f_i$ by a suitable subsequence if necessary, we see that   there is  a point $x$ in $\pi^{-1}\{y\}$, a smooth connected pointed curve $(D,d)$ over $k$,  and a covering set  $g_i:(D,d)\to (X,x)$. 

Now, to show that $X$ is geometrically-special over $k$, let $U\subset X$ be a dense open. Choose a dense open $U'\subset U$  such that $\pi_{U'}:U'\to \pi(U')$ is finite \'etale of degree $\deg \pi$. Then, we have shown that there is a point $x$ in $U'\subset U$, a smooth connected pointed curve $(D,d)$ over $k$, and a covering set  $g_i:(D,d)\to (X,x)$.  This concludes the proof. 
\end{proof}

Lemma \ref{lem:cw_geom_special} fits in well with Campana's conjectures, as the following proposition shows (included for the sake of completeness).
 
\begin{proposition}\label{thm:fin_et} Let $Y\to X$ be a finite \'etale morphism of   proper varieties over $k$.
Then the following statements hold.
\begin{enumerate}
\item The variety $Y$ is special over $k$ if and only if $X$ is special over $k$.
\item If $k=\CC$, then   $Y$ is Brody-special   if and only if $X$ is Brody-special.
\item   $Y$ is arithmetically-special over $k$ if and only if $X$ is arithmetically-special over $k$.
\item     $Y$ is geometrically-special over $k$ if and only if $X$ is geometrically-special over $k$.
\end{enumerate}
\end{proposition}

\begin{proof}  Note that $(1)$ follows from   \cite[Theorem~5.12]{Ca04}, and that  $(2)$ is obvious. Moreover, a standard ``Chevalley-Weil'' type argument shows that $X$ is arithmetically-special over $k$ if and only if $Y$ is arithmetically-special over $k$ (see for example \cite[Lemma~8.2]{JLitt}). This proves $(3)$.   Finally, $(4)$ follows from (the more general) Lemma \ref{lem:cw_geom_special}.
\end{proof}

\subsection{Products} It is not hard to see that if $X$ and $Y$ are arithmetically-special over $k$, then $X\times Y$ is arithmetically-special over $k$. Moreover, if $k=\CC$ and $X$ and $Y$ both admit a dense entire curve, say $f:\mathbb{C}\to X^{\an}$ and $g:\mathbb{C}\to X^{\an}$, then $(f,g):\mathbb{C}^2\to X\times Y$ has Zariski dense image. Thus, if  $\mathbb{C}\to \mathbb{C}^2$ is a metrically-dense entire curve, then the composed map $\mathbb{C}\subset \mathbb{C}^2\to X\times Y$ is a dense entire curve. Therefore,   if $X$ and $Y$ are Brody-special, then $X\times Y$ is Brody-special.
Finally, 
if $X$ and $Y$ are smooth projective special varieties over $k$, then $X\times Y$ is special by a theorem of Campana \cite{Ca04}. This ``product property'' also holds for geometrically-special varieties.

\begin{lemma}\label{lem:prod_geom_special} Let $X$ and $Y$ be geometrically-special varieties over $k$. Then $X\times Y$ is a geometrically-special variety over $k$.
\end{lemma}
\begin{proof} Let $g_{i}:(C_1,c_1)\to (X,x)$ and $h_i:(C_2,c_2)\to (Y,y)$ be a covering set   for $X$ and $Y$, respectively.   If $i\geq 1$ and $j\geq 1$, consider \[f_{i,j}:(C_1\times C_2, (c_c'))\to (X\times Y, (x,y)), \quad f_{i,j}(a,b) = (g_i(a),h_j(b)).\]
Note that the (countable) union of the graphs $\cup_{i,j} \Gamma_{f_{i,j}}$ of the morphisms $f_{i,j}$ is dense in $C_1\times C_2 \times X$.  Let $C\subset C_1\times C_2$ be a smooth connected very ample  divisor containing $(c,c')$. Then each morphism $f_{i,j}|_C:C\to X\times Y$ maps $(c,c') $ to $(x,y)$, and the union of their graphs is dense in $C\times X$. This shows that $X\times Y$ is geometrically-special over $k$.
\end{proof}
 
\subsection{Rationally connected varieties}
A  smooth projective variety over $k$ is \emph{rationally connected} if any two general points $p, q$ on $X$ can be connected by a chain of rational curves (see \cite[Definition~4.3]{DebarreBook})
Campana showed that such a variety is special (see \cite{Ca04}), and a recent theorem of Campana-Winkelmann shows that a rationally connected smooth projective variety is Brody-special.  It is straightforward to verify   that such a variety is geometrically-special, as we show now.

\begin{proposition}
If $X$ is a smooth projective rationally connected variety over $k$, then $X$ is geometrically-special over $k$.
\end{proposition}
\begin{proof}
Note that,  there is a dense open subset $U\subset X$  such that, for every $x$   in $U(k)$, the image of the   evaluation morphism   
\[
\mathrm{ev}_\infty\colon   \underline{\Hom}((\mathbb{P}^1_k,0),(X,x))\to   X, \quad f\mapsto f(\infty)
\] contains $U$, and is thus dominant.  By composing with an appropriate automorphism of $\mathbb{P}^1_k$, it follows that, for every $c\neq 0$ in $\mathbb{P}^1(k)$, the evaluation  morphism
\[
\mathrm{ev}_c\colon   \underline{\Hom}((\mathbb{P}^1_k,0),(X,x))\to   X, \quad f\mapsto f(c)
\]  is dominant.
  This implies that, for every $x$ in $U(k)$, there is a covering set $f_i:(\mathbb{P}^1_k,0)\to (X,x)$ for $X$, so that $X$   is geometrically-special over $k$, as required.
\end{proof}

It is not known whether a rationally connected smooth projective variety is arithmetically-special. In fact, this is open even for Fano varieties; see  \cite{HassettSurvey} for a survey of some known results. 

\subsection{The case of curves}
The conjectures of Campana and Lang are fully understood in the case of curves (unless one looks at orbifold curves). We state part of the optimal statement  below. Note that we do not use this result, and that it is superseded by    more general results on closed subvarieties of abelian varieties (see Theorem \ref{thm:ab_var}).

\begin{theorem} Let $X$ be a smooth  projective   curve over $k$. Then the following statements are equivalent.
\begin{enumerate}
\item The curve $X$ is of  general type over $k$ (i.e., $\mathrm{genus}(X) > 1$).
\item The curve $X$ is not special over $k$.
 \item The curve $X$ is geometrically hyperbolic over $k$.
\item The curve $X$ is not geometrically special over $k$.
\end{enumerate}
 \end{theorem}
 \begin{proof}  The equivalence of $(1)$ and $ (2)$ for curves is well-known.   The other implications follow from the definitions and the finiteness theorem of De Franchis (which says that, given  smooth projective  curves $C$ and $D$ with $D$ of  general type, the set of non-constant morphism $C\to D$ is finite.)
 \end{proof}

 \subsection{Reducing to the field of complex numbers} Let $L/k$ be an extension of algebraically closed fields of characteristic zero.
 If $X$ is geometrically-special over $k$, then $X_L$ is obviously geometrically-special over $L$. The converse is probably true, but not clear to us.   
 
Related to this is the following simple lemma which allows us to reduce certain proofs to varieties over $\mathbb{C}$ (see, for example, our proof of Theorem \ref{thm:main_result_I}).
 
 \begin{lemma}\label{lem:descent}
Let $X$ be a geometrically-special variety over $k$. Then, there is a countable algebraically closed subfield $k_0\subset k$, and a variety $X_0$ over $k_0$ with an isomorphism $X_{0,k}\cong X$  such that $X_{0}$ is geometrically-special over $k_0$.
 \end{lemma}
\begin{proof} 
 By definition, there is a dense subset $S'\subset X(k)$ such that, for every $x$ in $S'$, there is a pointed curve $(C_x,c_x)$ (which depends on $x$) and a covering set $(f_{x,i}:(C_x,c_x)\to (X,x))_{i}$.  Let $S\subset S'$ be a countable dense subset   (see Lemma \ref{lem:countable} below).  
 
For each $x \in S$, we choose a countable algebraically closed subfield $k_x\subset k$ such that $(C_{x}, c_{x})$ can be defined over $k_x$, the variety $X$ can be defined over $k_x$, the point $x$ can be defined over $k_x$, and the covering set $(f_{i,x}:(C_{x},c_x)\to (X,x))_i$ can be defined over $k_x$. (Note that $k_x$ might not have finite transcendence degree, because we are descending a countably infinite amount of data.) Let $k_0$ be the algebraic closure (in $k$) of the union $\cup_x k_x$ of the countably many countable subfields $k_x$ of $k$. Note that $k_0$  is countable of characteristic zero and that $X$ has a model, say $X_0$, over $k_0$ which by construction is geometrically-special over $k_0$.
\end{proof}

\begin{lemma}\label{lem:countable}
Let $X$  be a variety over $k$, and let $S\subset X(k)$  be a dense subset. Then $S$ contains a countable dense subset.
\end{lemma}
\begin{proof} Let $d:=\dim(X)$. For any closed subset $Z\subset X$, let the $i$-th entry of $d(Z)\in \mathbb{N}_0^{d+1}$ denote the number of irreducible components of $Z$ of dimension $i$. If $Z'$ is   closed in $X$ and  $Z\subsetneq Z'$, then $d(Z) \prec d(Z')$ with $\prec $ the lexicographic order on $\mathbb{N}^{d+1}_0$.

Let  $\mathcal{N}:= \{d(\overline{\Sigma}) \ | \ \Sigma \subset S \ \textrm{countable}\}$. For each $n\in \mathcal{N}$, choose a countable subset $\Sigma_n\subset S$ with $d(\overline{\Sigma_n}) =n$. Let $A:= \overline{\cup_n \Sigma_n}$. Then,  as a countable union of countable sets is countable, we have that  $d(A)\in \mathcal{N}$. Also, for every $n$ in $\mathcal{N}$, we have
$$
n= d(\overline{\Sigma_n})\prec d(A).
$$ If $A\neq X$, then there is an $x\in (X\setminus A)(k)$ which also lies in $S$, so that $d(A) \prec d(A\cup\{x\})\in \mathcal{N}$. This is a contradiction. Therefore,  we conclude that $A=X$, so that  $\cup_n \Sigma_n$ is  dense in $X$, as required.
\end{proof}

\section{Albanese maps}\label{section:abelian varieties}
 
In this section we prove that the Albanese map of a geometrically-special normal projective variety is surjective, has connected fibers, and has no multiple fibers in codimension one; see Theorem \ref{thm:main_result_I} for a precise statement.

\subsection{Abelian varieties are special}  We start by showing    that abelian varieties are special in any sense of the word ``special''.  We give a bit of an "overkill" proof of the fact that abelian varieties are geometrically-special (see Remark \ref{remark:referee}) for a more elementary proof.

\begin{proposition}\label{prop:ab_var_are_special}
If $A$ is an abelian variety over $k$, then $A$ is special over $k$, arithmetically-special over $k$, and geometrically-special over $k$.
\end{proposition}
\begin{proof}
The fact that $A$ is special over $k$ was shown by Campana \cite{Ca04}. By a theorem of Frey--Jarden \cite{FreyJarden}   (see   \cite[\S 3.1]{JAut}),  there is a   subfield $K\subset k$ which is a finitely generated field extension of $\mathbb{Q}$ and a model $\mathcal{A}$ for $A$ over $K$  such that $\mathcal{A}(K)$ is dense in $A$. This shows that $A$ is arithmetically-special over $k$. Finally, to conclude the proof, let us show that $A$ is geometrically-special over $k$.

First, by Poincar\'e's irreducibility theorem, the abelian variety $A$ is isogenous to a finite product of simple abelian varieties. Therefore, by Lemma \ref{lem:cw_geom_special} and Lemma \ref{lem:prod_geom_special}, we may and do assume that $A$ is simple. To prove the proposition, let $0$ denote the origin of $A$. Since any point of $A(k)$ can be translated to $0$,  it suffices to show that there is a smooth connected pointed curve $(C,c)$ and a covering set  $f_i:(C,c)\to (A,0)$. 

If $\dim A=1$ (so that $A$ is an elliptic curve), define $C:=A$, $c=0$, and note that the graphs $\Gamma_n$ of  the morphism (multiplication by $n$) $[n]:C\to A$ form a covering set which send $0$ to $0$.

If $\dim A\geq 2$, let $C\subset A$ be a smooth irreducible curve containing $c:=0$. Since $A$ is simple, we have that $C$ is of genus at least two. In particular,  by Raynaud's theorem (\emph{formerly} the Manin-Mumford conjecture) \cite{OesterleBB, Raynaud}, almost all points of $C(k)$ are non-torsion. As $A$ is simple, any non-torsion point of $A$ is non-degenerate (i.e., the subgroup generated by such a point is dense in $A$). Now, define $f_n:C\to A$ by $f_n(c) = n\cdot c$. Then $f_n(0) =0$ and, as $C$ contains a dense set of non-degenerate points, the union of the graphs $\Gamma_{f_n}$ is dense in $C\times A$.  This shows that $A$ is geometrically-special over $k$, as required.
\end{proof}

\begin{remark}[Referee]\label{remark:referee}
A referee has pointed out to us that one does not need Manin-Mumford to prove that an abelian variety $A$ is geometrically-special. Indeed, in the proof above, we only use that there exists a curve $C\subset A$ which contains infinitely many non-torsion points of $A$. To construct such a curve, one can use that the height of the set of torsion points on $A$ is bounded (with respect to any fixed embedding of $A $ into a projective space).
\end{remark} 

\begin{proposition}\label{prop:ab_var_are_special2}
 If $A$ is an abelian variety over $\CC$, then $A$ is Brody-special.
\end{proposition}
\begin{proof}
  Let $\mathbb{C}\to \mathbb{C}^{\dim A}$ be a  metrically-dense holomorphic map. If $\mathbb{C}^{\dim A}\to A^{\an}$ is the exponential, then the image of the composed   map $\mathbb{C}\to \mathbb{C}^{\dim A}\to A^{\an}$   is  Zariski-dense.
\end{proof}

\subsection{Surjectivity of the Albanese map} 
For a closed subvariety $V$ of an abelian variety $A$, we let $\mathrm{Sp}(V)$  be the union of translates of positive-dimensional abelian subvarieties of $A$ contained in $V$. Recall Kawamata-Ueno's theorems that $\mathrm{Sp}(V)$ is a closed subscheme of $V$, and that $\mathrm{Sp}(V) \neq V$ if and only if $V$ is of general type. We now deduce from Yamanoi's work that $V$ is geometrically hyperbolic modulo $\mathrm{Sp}(V)$.    

\begin{theorem}[Yamanoi]\label{thm:V_is_mod} Let $A$ be an abelian variety over $k$.
Let $X\subset A$ be a closed subvariety of $A$. Then $X$ is geometrically hyperbolic modulo $\mathrm{Sp}(X)$.
\end{theorem}

\begin{proof} 
Let $\Delta^{geom-hyp}_X$ be the intersection over all closed subsets $\Delta\subset X$ such that $X$ is geometrically hyperbolic modulo $\Delta$ (as defined in \cite[Section~11]{JBook}). Then,   by  Yamanoi's theorem \cite[Theorem~13.1]{JBook} (which is a consequence of  the original \cite[Corollary~1.1 and Corollary~1.3]{Yamanoi1}), we have that  $\mathrm{Sp}(X) = \Delta_X^{geom-hyp}$.   
\end{proof}

\begin{theorem}[Bloch-Ochiai-Kawamata, Faltings, Ueno, Yamanoi]\label{thm:ab_var}
Let $A$ be an abelian variety over $k$. Let $X\subset A$ be a closed integral subvariety. Then the following are equivalent.
\begin{enumerate}
\item The variety $X$ is the translate of an abelian subvariety of $A$.
\item The variety $X$ is special.
\item If $k=\CC$, the variety $X$ is Brody-special.
\item The projective variety $X$ is arithmetically-special over $k$.
\item The variety $X$ is geometrically-special over $k$.
\end{enumerate}
\end{theorem}
\begin{proof}  The implications $(1)\implies (2)$, $(1)\implies (3)$, $(1)\implies (4)$, and $(1)\implies (5)$ follow from Proposition \ref{prop:ab_var_are_special} and Proposition \ref{prop:ab_var_are_special2}.

We   show that $(5)\implies (1)$.
Assume $(5)$ holds,  and consider the   Ueno fibration $f:X\to Y$ (see  \cite[Chapter~IV, Theorem~10.9]{Ueno}); recall that this is a surjective morphism   whose fibres are abelian varieties and that  $Y$ is a variety of general type which can be embedded into an    abelian variety. Since $Y$ is of general type and can be embedded into an abelian variety, it follows from Kawamata-Ueno's theorem that $\mathrm{Sp}(Y)\neq Y$,  so that  Yamanoi's theorem (Theorem \ref{thm:V_is_mod}) implies that $Y$ is pseudo-geometrically hyperbolic.   However, since $X$ is geometrically-special and dominates the pseudo-geometrically hyperbolic variety $Y$, we have that $Y$ is zero-dimensional (Corollary \ref{cor:image_is_not_psgeom}).    Therefore, $X$ must be the translate of an abelian subvariety of $A$.

The fact that $(2)\implies (1)$ follows directly from Ueno's fibration theorem. One can show $(3)\implies (1)$ and $(4)\implies (1)$ by replacing Yamanoi's theorem above with    Bloch-Ochiai-Kawamata \cite{Kawa80} and Faltings \cite{Fa94}, respectively.
\end{proof}

\begin{corollary}\label{cor:surjectivity_alb} If  $X$ is a geometrically-special normal proper   variety over $k$, then the Albanese map $X\to \mathrm{Alb}(X)$ is surjective.
\end{corollary}
\begin{proof}
Note that the image of $X$ in $\mathrm{Alb}(X)$ is   geometrically-special (Lemma \ref{lem:surj}). Moreover, by Theorem \ref{thm:ab_var}, any geometrically-special closed (integral) subvariety of $\mathrm{Alb}(X)$ is the translate of an abelian subvariety. This implies that the image of $X\to \mathrm{Alb}(X)$ is  the translate of an abelian subvariety which, by the universal property of Albanese maps, must equal $\mathrm{Alb}(X)$.  
\end{proof}

An important structure result for the Albanese map of a geometrically-special variety in the study of its fundamental group is the following corollary. 

 \begin{corollary} \label{cor:covers_have_surj_alb}
 If $X$ is a geometrically-special normal projective variety over $k$ and $Y\to X$ is a finite \'etale morphism, then the Albanese map of $Y$ is surjective.
 \end{corollary}
 \begin{proof}
 By Lemma \ref{lem:cw_geom_special}, we have that $Y$ is geometrically-special. Since $Y$ is a geometrically-special normal projective   variety, its Albanese map is surjective (Corollary \ref{cor:surjectivity_alb}).  
 \end{proof}

 Note that Corollary \ref{cor:covers_have_surj_alb} implies that the augmented irregularity of a geometrically-special normal projective variety over $k$ is bounded from above by $\dim(X)$. 
 
\subsection{Connectedness of the fibres}\label{section:connectedness}
By using Kawamata's extension of the Ueno fibration theorem for closed subvarieties of abelian varieties to finite covers of abelian  varieties \cite[Theorem~23]{KawamataChar} and Yamanoi's recent work on the hyperbolicity of covers of abelian varieties, we can show that the Albanese map of a geometrically-special variety has connected fibres. We start with stating Kawamata's result.

\begin{theorem}[Kawamata]\label{thm:kawamata_fibn}
Let $X$ be a   normal projective  variety over $k$ and let $X\to A$ be a finite morphism. Then  the following data exist.
\begin{enumerate}
\item An abelian subvariety $B$ of $A$;
\item finite \'etale Galois covers $X'\to X$  and $B'\to B$;
\item  a normal projective variety $Y$ over $k$ of general type;
\item    a finite morphism $Y\to A/B$ with $A/B$ the quotient of $A$ by $B$ such that  $X'$ is a fiber bundle over $Y$ with fibers $B'$ and with translations by $B'$ as structure group. 
\end{enumerate}  
\end{theorem}
\begin{proof}
See \cite[Theorem~23]{KawamataChar}.
\end{proof}

We follow Yamanoi and say  that a   normal projective  variety over $k$ is of \emph{maximal Albanese dimension} if the Albanese morphism   $Y\to\mathrm{Alb}(Y)$ is generically finite onto its image. The additional ingredient we need for proving the connectedness of the fibres of the Albanese map is Yamanoi's work on the hyperbolicity of  varieties with maximal Albanese dimension (which builds on Theorem \ref{thm:V_is_mod}) .

\begin{theorem} \label{thm:yamanoi}  
Let $Y$ be a   normal projective variety of general type over $k$. If $Y$  is of maximal Albanese dimension, then $Y$ is pseudo-algebraically hyperbolic and pseudo-geometrically hyperbolic over $k$.
\end{theorem}
\begin{proof}  
 As a pseudo-algebraically hyperbolic projective variety is pseudo-geometrically hyperbolic   (Proposition \ref{prop:alg_hyp_implies_geom_hyp}), it suffices to show that $Y$ is pseudo-algebraically hyperbolic over $k$. To do so,  by a standard specialization argument (see for instance \cite[Lemma~9.2]{vBJK}), we may and do assume that $k=\CC$. Then, the statement follows from  \cite[Corollary~1.(1) and Corollary~1.(3)]{Yamanoi1}.
\end{proof}

 \begin{corollary}\label{cor:yamanoi_special_covers}
Let $X$ be a   normal proper   variety  over $k$ of maximal Albanese dimension. Then the following statements are equivalent.
\begin{enumerate}
\item The Albanese map $a:X\to \mathrm{Alb}(X)$ is birational.
\item The variety $X$ is special.
\item The variety $X$ is geometrically-special over $k$. 
\end{enumerate}  
 \end{corollary}
 \begin{proof} By Proposition  \ref{prop:ab_var_are_special} and the birational invariance of geometric-specialness (resp. specialness) (see Lemma \ref{lem:surj}), we have that $(1)\implies (2)$ and $(1) \implies (3)$.  
 
 Now, assume  that $X$ is geometrically-special  (resp. special).
 Then, by Corollary  \ref{cor:surjectivity_alb}  (resp. Theorem \ref{thm:alb_of_special}), the Albanese map  $a:X\to \mathrm{Alb}(X)$  is surjective. Let $X\to Z\to \mathrm{Alb}(X)$ be its Stein factorization; note that $Z$ is a normal variety over $k$ and that $Z\to \mathrm{Alb}(X)$ is a finite surjective morphism. Since $X$ is geometrically-special (resp. special) and $X\to Z$ is surjective, it follows from Lemma \ref{lem:surj}  that $Z$ is geometrically-special (resp. special) . 
 We now apply Kawamata's fibration theorem (Theorem \ref{thm:kawamata_fibn}) and see that the following data exists.
 \begin{enumerate}
\item An abelian subvariety $B$ of $A$;
\item finite \'etale Galois covers $Z'\to Z$  and $B'\to B$;
\item  a normal projective variety $Y$ of general type over $k$;
\item    a finite morphism $Y\to A/B$ with $A/B$ the quotient of $A$ by $B$ such that  $Z'$ is a fiber bundle over $Y$ with fibers $B'$ and with translations by $B'$ as structure group.
\end{enumerate}  
Since $Y$ is of general type, it follows from Yamanoi's theorem (Theorem \ref{thm:yamanoi}) that $Y$ is pseudo-geometrically hyperbolic over $k$.

Now, since $Z$ is geometrically-special (resp. special) and $Z'\to Z$ is finite \'etale, it follows from Proposition \ref{thm:fin_et}  that $Z'$ is geometrically-special (resp. special).
 Since $Z'$ is geometrically-special (resp. special) and surjects onto the variety $Y$, it follows from Lemma \ref{lem:surj}  that $Y$ is also geometrically-special (resp. special).  Thus, since $Y$ is a geometrically-special   pseudo-geometrically hyperbolic projective variety (resp. special variety of general type), we conclude that $Y$ is zero-dimensional. This implies that $Z'=B'$ is an abelian variety, so that the composed (finite surjective) morphism $Z'\to Z\to\mathrm{Alb}(X)$ is finite \'etale. We conclude that $Z\to \mathrm{Alb}(X)$ is finite \'etale, so that $Z$ is an abelian variety \cite[Section~IV.18]{MumAb}. It now follows from the universal property of the Albanese map that $Z=\mathrm{Alb}(X)$, so that the morphism $X\to \mathrm{Alb}(X)$ has connected fibres. Finally, since $X$ has maximal Albanese dimension and the morphism $X\to \mathrm{Alb}(X)$ is surjective with connected fibres, we conclude that it is birational. This proves that $(3)\implies (1)$ (resp. $(2)\implies (1)$).
 \end{proof}
 
\begin{corollary}\label{cor:geom_special_albanese}  
If $X$ is a geometrically-special normal projective variety over $k$, then the Albanese map $X\to \mathrm{Alb}(X)$ is surjective  and has  connected fibres.
\end{corollary}
\begin{proof}  
 Let $a:X\to \mathrm{Alb}(X)$ be the Albanese map. Since $X$ is geometrically-special, the (proper) morphism $a$ is surjective (Corollary \ref{cor:surjectivity_alb}). Let $X\to Z\to \mathrm{Alb}(X)$ be its Stein factorization; note that $Z$ is a normal geometrically-special projective variety over $k$ and that $Z\to \mathrm{Alb}(X)$ is a finite surjective morphism. Clearly, $Z$ has maximal Albanese dimension, so that $Z\to \mathrm{Alb}(X)$ is birational by Corollary \ref{cor:yamanoi_special_covers}, and thus an isomorphism by Zariski's Main Theorem, as required.
\end{proof}
 
To complete the proof of Theorem \ref{thm:main_result_I}, it remains to show that the Albanese map $X\to \mathrm{Alb}(X)$ of a geometrically-special variety $X$ has no multiple fibers in codimension one.   To prove this fact, we   introduce  geometrically-special orbifold pairs and the orbifold extension of Demailly's notion of algebraic hyperbolicity.

\subsection{Orbifold pairs}  
   A $\QQ$-divisor $\Delta$ on a smooth projective   variety $X$ over $k$ is an \emph{orbifold divisor (on $X$)} if $\Delta = \sum_{i=1}^r (1-\frac{1}{m_i})\Delta_i$, where $m_1,\ldots, m_r \geq 1$ are rational numbers    and $\Delta_1,\ldots, \Delta_r$ are prime  divisors; the rational number $m_i$ is referred to as the \emph{multiplicity of $\Delta_i$ in $\Delta$}.
A pair $(X,\Delta)$ with $X$ a smooth projective variety over $k$ and   $\Delta = \sum_{i=1}^r a_i \Delta_i$ a $\mathbb{Q}$-divisor with $0\leq a_i\leq 1$ is an \emph{orbifold (or orbifold-pair)}. Note that $\Delta$ is an orbifold divisor on $X$ in this case. 

Let $(X,\Delta_X)$ and $(Y,\Delta_Y)$ be orbifold pairs. We let   $m_X$ and $m_Y$ be the multiplicities of $\Delta_X$ and $\Delta_Y$, respectively. (In particular, for a prime divisor $\Delta_i\subset X$ contained in the support of $\Delta_X$, the rational number $m_X(\Delta_i) $ is the multiplicity of $\Delta_i$  in $\Delta_X$.) We define an \emph{orbifold morphism} $(Y,\Delta_Y) \to (X,\Delta_X)$ to be a morphism $f:Y\to X$ of varieties over $k$  such that $f(Y)\not\subset \lceil \Delta_X \rceil$ and,  for every prime divisor $D$ on $X$ and  prime divisor $E$ on $Y$, the inequality  $t \cdot m_Y(E) \geq m_X(D)$ holds, where $t$ is defined by the relation $f^*(D)=t \cdot E+R$ with $R$ an effective divisor not containing $E$.  

Following standard conventions, we will identify a smooth projective  variety $X$ with the orbifold $(X,0)$, where $0$ is the empty divisor.

In \cite{Ca11}, Campana extended the notions of special variety (resp. Brody-special variety) to the setting of orbifold pairs, and then  conjectured that an orbifold pair $(X,\Delta)$ over $\CC$ is  Brody-special if and only if it is special. 
He also extended the notion of   arithmetically-special variety to the orbifold setting, and formulated the orbifold analogue of his arithmetic conjecture  (Conjecture \ref{conj:campana_ar}).    In this section we  extend the notions of geometrically-special and pseudo-geometric hyperbolicity  to the setting of orbifold pairs.

 \begin{definition}
 An orbifold pair $(X,\Delta)$ is \emph{geometrically-special over $k$} if,  for every dense open subset $U\subset X$, there exists a smooth projective connected curve $C$ over $k$, a point $c$ in $C(k)$, a point $u$ in $U(k)\setminus \Delta$, and a sequence of pairwise distinct orbifold morphisms $f_i:C\to (X,\Delta)$ with $f_i(c) = u$ for $i=1,2,\ldots$ such that $C\times X $ is covered by the graphs $\Gamma_{f_i}$ of these maps, i.e., the closure of $\cup_{i=1}^{\infty} \Gamma_{f_i}$  in $C\times X$ equals $ C\times X$.  (As mentioned above,  we identify $C$ with the orbifold $(C,0)$.)
 \end{definition}
 
 \begin{remark}  Let $X$ be a smooth projective connected variety over $k$ and let $D$ be an integral divisor. For $m\in \mathbb{Z}_{\geq 1}$, 
 define $\Delta_m := (1-\frac{1}{m})D$ and  $X_m := (X,\Delta_m)$. Then the orbifold pair $X_1$ is geometrically-special over $k$ if and only if $X$ is geometrically-special.    Moreover, for all $m\geq 1$, we have that
    \[
 X_m \ \textrm{is} \ \textrm{geometrically-special}   \implies X_{m-1} \ \textrm{is} \ \textrm{geometrically-special}  
    \]
 
 \end{remark}
 
 By adapting the  proof of Lemma \ref{lem:surj0} to the orbifold setting, one obtains the following \emph{stronger} statement. 
 
 \begin{lemma}\label{lem:surj2}
 Let $(X,\Delta)\to (X',\Delta')$ be a surjective morphism of orbifold pairs. If $(X,\Delta)$ is geometrically-special over $k$, then $(X',\Delta')$ is geometrically-special over $k$. \qed
 \end{lemma}

  A geometrically-special orbifold pair is as far away as possible from being ``pseudo-geometrically hyperbolic'' in the following sense  (generalizing Definition \ref{def:psgeomhyp2}).

 \begin{definition}
If  $(X,\Delta)$ is an orbifold pair over $k$ and $Z\subset X$ is a closed subset containing $\Delta$, then $(X,\Delta)$ is \emph{geometrically hyperbolic modulo $Z$}  if, for every $x$ in $X(k)\setminus Z$, every smooth projective connected curve $C$ over $k$, and every  $c$ in $C(k)$, the set of orbifold morphisms $f:C\to (X,\Delta)$ with $f(c) = x$ is finite. We say that an orbifold pair $(X,\Delta)$ over $k$ is \emph{pseudo-geometrically hyperbolic over $k$} if there is a proper closed subset $Z\subsetneq X$ containing $\Delta$ such that $(X,\Delta)$ is geometrically hyperbolic modulo $Z$.
 \end{definition}

 \subsection{Pseudo-algebraic hyperbolicity}\label{section:psalg} We have already seen the relation between pseudo-geometric hyperbolicity and pseudo-algebraic hyperbolicity (see Proposition \ref{prop:alg_hyp_implies_geom_hyp}). In this section we extend this relation to the orbifold setting following \cite{Rou10, Rou12}.  

\begin{definition} Let $(X,\Delta)$ be an orbifold pair, and let $E\subset X$ be a closed subset. Then    $(X,\Delta)$  is \emph{algebraically hyperbolic modulo $E$ over $k$} if, for every ample line bundle $L$ on $X$, there is a constant $\alpha_{X,\Delta,L,E}$ such that, for every smooth projective connected curve $C$ over $k$, and every orbifold morphism $f:C\to (X,\Delta)$ with $f(C)\not\subset \Delta\cup E$, the inequality
\[
\deg_C f^\ast L \leq \alpha_{X,\Delta,L,E} \cdot \mathrm{genus}(C)
\] holds.
\end{definition}

\begin{remark}[Independence of choice of ample bundle]\label{remark:independence_of_choice_of_ample0}
Let $(X,\Delta)$ be an orbifold pair over $k$, and let $E\subset X$ be a closed subset. Then  $(X,\Delta)$ is algebraically hyperbolic modulo $E$ over $k$ if and only if there exists an ample line bundle $L$ on $X$ and a constant $\alpha_{X,\Delta,L,E}$ such that, for every smooth projective connected curve $C$ over $k$, and every orbifold morphism $f:C\to (X,\Delta)$ with $f(C)\not\subset E$,  the inequality
$
\deg_C f^\ast L \leq \alpha_{X,\Delta,L,E} \cdot \mathrm{genus}(C).
$  holds.
\end{remark}

\begin{remark}[A big line bundle suffices]\label{remark:independence_of_choice_of_ample}
 An orbifold pair $(X,\Delta)$ over $k$ is pseudo-algebraically hyperbolic over $k$ if and only if there exist a big line bundle $L$ on $X$,  a proper closed subset $E\subset X$, and  a constant $\alpha_{X,\Delta,L,E}$ such that, for every smooth projective connected curve $C$ over $k$, and every orbifold morphism $f:C\to (X,\Delta)$ with $f(C)\not\subset E$,  the inequality
$
\deg_C f^\ast L \leq \alpha_{X,\Delta,L,E} \cdot \mathrm{genus}(C)
$  holds. 
To prove this, write $L= A + F$ with $A$ ample and $F$ effective. Note that $F$ contains the augmented base locus of $L$. Then, by Remark \ref{remark:independence_of_choice_of_ample0}, we have that $(X,\Delta)$ is algebraically hyperbolic modulo the proper closed subset $E\cup F$.
\end{remark}

 The following  proposition   generalizes Proposition \ref{prop:alg_hyp_implies_geom_hyp} to the setting of orbifold pairs under the additional assumption that the rational curves on the space $X$ are contained in a proper closed subset.

\begin{proposition} \label{prop:from_alg_hyp_to_geom_hyp}
Let $(X,\Delta)$ be an orbifold pair and let $E\subset X$ be a closed subset such that every non-constant morphism $\mathbb{P}^1_k\to X$ factors over $E$. If  $(X,\Delta)$ is algebraically hyperbolic modulo $E$, then $(X,\Delta)$ is geometrically-hyperbolic modulo $E$.
\end{proposition}
\begin{proof} We may and  do assume that $E\neq X$. Suppose that $(X,\Delta)$ is not geometrically-hyperbolic modulo $E$. Then, we can choose   a smooth proper connected curve $C$ over $k$, a point $c\in C(k)$, a point $x\in X(k)\setminus E\cup \Delta$, and  an infinite sequence of orbifold morphisms $f_i:C\to (X,\Delta)$ with $f_i(c) = x$. Note that, as $x\in X(k)\setminus E$, we have that $f_i(C)\not\subset E$. Therefore, if $L$ is an ample line bundle, since $X$ is algebraically-hyperbolic modulo $E$, there is a real number $\alpha$ depending only on $X, \Delta, E, L$ such that, for every $i=1,2,\ldots$, the inequality
\[
\deg_C f_i^\ast L \leq \alpha \cdot \mathrm{genus}(C)
\] holds. Thus, replacing $(f_i)_{i=1}^\infty$ by a suitable subsequence if necessary, we may and do assume that $d:=\deg_C f_i^\ast L$ is independent of $i$. Then, the component $\underline{\Hom}_k^d((C,c),(X,x))$ parametrizing 
 pointed morphisms $(C,c)\to (X,x)$ over $k$  of degree $d$ is a positive-dimensional variety over $k$ (as it is quasi-projective and has infinitely many  elements). This implies by Mori's bend-and-break that $X$ has a rational curve passing through $x$ (see \cite[Proposition~3.1]{DebarreBook}), contradicting the fact that every rational curve in $X$ is contained in $E$.
\end{proof}

\begin{definition}
An orbifold pair $(X,\Delta)$ is \emph{pseudo-algebraically hyperbolic over $k$} if there is a proper closed subset $E\subsetneq X$ containing $\Delta$ such that $(X,\Delta)$ is algebraically  hyperbolic modulo $E$.
\end{definition}

 \subsection{Lang's conjecture for orbifold divisors on abelian varieties}\label{section:yamanoi_orb}

  The following result of Yamanoi confirms  the  expectation that abelian varieties endowed with a big orbifold divisor are pseudo-algebraically hyperbolic in the orbifold sense.

 \begin{theorem}[Yamanoi]\label{thm:yamanoi_erwan}
 Let $\Delta$ be a big orbifold divisor on an abelian variety $A$ over $k$.     Then, the orbifold pair  $(A,\Delta)$   is pseudo-algebraically hyperbolic over $k$.
 \end{theorem}
  
 \begin{proof}  We may and do assume that $k=\CC$. Let us show how the result follows from Yamanoi's theorem \cite[Corollary~2]{Yamanoi1}.  
 
 Write  $\Delta=\sum_{i=1}^r \left(1-\frac{1}{m_i}\right) D_i$, where  $m_1,\ldots, m_r> 1$ are  rational numbers and $D_1,\ldots, D_r$ are integral divisors, respectively.
   Let $L$ be an ample line bundle on $A$. Now, since $\Delta$ is big, we may choose a rational number $\epsilon >0$ and an effective divisor $F$   such that $\Delta-F-\epsilon L$ is ample, so that $\Delta - \epsilon L$ is big. 
   
   With these choices made, we now apply Yamanoi's theorem \cite[Corollary~2]{Yamanoi1}  to the integral  divisors $D_i$ and see that, for every $i=1,\ldots, r$, there is a proper closed subset $E_i:=\Sigma(A,D_i,L)$ and a positive real number $a_i = a_i(A,D_i,L)$ such that, for every smooth projective connected curve $C$ over $k$ and every morphism of orbifold pairs $f:C\to (A,\Delta)$ with $f(C) \not\subset   E_i\cup \Delta$, the inequality
 \[
 \deg f^\ast D_i \leq \# f^{-1}(D_i) + a_i \cdot \mathrm{genus}(C) + \frac{\epsilon}{r} \deg f^\ast L
 \] holds.  
Since $f:C\to (A,\Delta)$ is a morphism of orbifold pairs, we have that 
\[\# f^{-1}(D_i) \leq \frac{1}{m_i} \deg f^\ast D_i.\]
This implies, for every $i=1,\ldots, r$,   that
\[
\deg f^\ast D_i \leq \frac{1}{m_i}\deg f^\ast D_i + a_i \cdot \mathrm{genus}(C) + \frac{\epsilon}{r} \deg f^\ast L.
\]   Thus, for every morphism  of orbifold pairs $f:C\to (A,\Delta)$ with $f(C) \not\subset   \cup_{i=1}^r E_i\cup \Delta$, 
\[
\sum_{i=1}^r \left(\deg f^\ast D_i  -  \frac{1}{m_i}\deg f^\ast D_i\right) \leq  \sum_{i=1}^r \left(a_i \cdot \mathrm{genus}(C) + \frac{\epsilon}{r} \deg f^\ast L\right).
\]
Defining $a:=\sum_{i=1}^r a_i$, we obtain that
\[
  \deg f^\ast(\Delta -\epsilon L) \leq  a\cdot \mathrm{genus}(C).
\] Since $\Delta-\epsilon L$ is big, we conclude that $(A,\Delta)$ is algebraically hyperbolic modulo $E$ (Remark \ref{remark:independence_of_choice_of_ample}), as required.  
 \end{proof}

 \begin{theorem}[Yamanoi + $\epsilon$]\label{thm:yamanoi_erwan2}
  Let $\Delta$ be a big orbifold divisor on an abelian variety $A$ over $k$. Then, the orbifold pair $(A,\Delta)$   is pseudo-geometrically-hyperbolic over $k$.  
 \end{theorem}
 \begin{proof}  By Yamanoi's theorem (Theorem \ref{thm:yamanoi_erwan}), we have that $(A,\Delta)$ is pseudo-algebraically hyperbolic over $k$. Therefore, as $A$ has no rational curves, the result follows from Proposition \ref{prop:from_alg_hyp_to_geom_hyp}.
 \end{proof}

 Theorem \ref{thm:conj_gt_holds} below  provides a converse to Theorem \ref{thm:yamanoi_erwan2} and may be considered as a version of Lang's conjecture (Conjecture \ref{conj:lang}) for orbifold pairs on abelian varieties. We include it for the sake of future reference.
 
 \begin{lemma}\label{lem:orbifold_grouplessness}
 Let $(X,\Delta)$ be an orbifold pair over $k$, and let $E\subset X$ be a proper closed subset. If $(X,\Delta)$  is geometrically hyperbolic modulo $E$  over $k$ and  $A$ is an abelian variety over $k$,  then every non-constant orbifold morphism $A\to (X,\Delta)$ factors over $E$.
 \end{lemma}
 \begin{proof}
 Let $\varphi:A\to (X,\Delta)$ be a non-constant orbifold morphism. Let $0\in A$ be the origin and assume $\varphi(0)\not \in E$. Let $X'= \varphi(A)$, and note that $X'$ is geometrically hyperbolic modulo the proper closed subset $E':=E\cap X'$. Since $A$ is geometrically-special (Proposition \ref{prop:ab_var_are_special}), it follows that $X'$ is geometrically-special (Lemma \ref{lem:surj0}). Since $X'$ is geometrically-special and pseudo-geometrically hyperbolic, we conclude that $X'$ is zero-dimensional (Corollary \ref{cor:image_is_not_psgeom}), so that $\varphi$ is constant, as required.
 \end{proof}

 \begin{theorem}\label{thm:conj_gt_holds}
 Let $\Delta$ be an orbifold divisor on an abelian variety $A$ over $k$. Then, the orbifold pair $(A,\Delta)$ is of general type if and only if $(A,\Delta)$ is pseudo-geometrically hyperbolic over $k$.
 \end{theorem}
 \begin{proof}
Since $\omega_A=\mathcal{O}_A$, the orbifold pair $(A,\Delta)$ is of general type  if and only if $\Delta$ is big. Thus, if $(A,\Delta)$ is of general type, then $(A,\Delta)$ is pseudo-geometrically hyperbolic over $k$ by Theorem \ref{thm:yamanoi_erwan2}.

 Conversely, assume that $(A,\Delta)$  is not of general type.  Let $E\subset A$  be a closed subset such that $(A,\Delta)$ is geometrically hyperbolic modulo $E$. To conclude the proof, it suffices to show that $E=A$.

 Note that $\Delta$ is not big. Let $\Delta'$ be a non-big   component of $\Delta$. (As the sum of big divisors is big, such a component exists.) Then, we have a non-zero abelian variety $B$, a homomorphism $f:A\to B$, and a big orbifold divisor $D$ on $B$ such that $\Delta'\subset f^{-1}(\mathrm{Supp}(D))$.  Thus, for every $t$ in $B\setminus D$, the fibre $A_t$ of $f:A\to B$ is a positive-dimensional abelian subvariety of $A$. Note that, as $A_t$ and $\Delta'$ are  disjoint, we have that  the  natural inclusion $A_t\to A\setminus \Delta'$ induces an orbifold inclusion $A_t\subset (A,\Delta')$.  Since $(A,\Delta')$ is geometrically hyperbolic modulo $E$, by Lemma \ref{lem:orbifold_grouplessness}, the inclusion $A_t\subset A$ factors over $E$. This implies that the induced morphism $E\to B$ is dominant with general fibres of dimension $\dim A_t$, so that $\dim E = \dim A_t + \dim B =\dim A$. This implies that $E=A$, as required.
 \end{proof}

 \subsection{The Albanese map has no multiple fibres}
In this section we complete the proof of our   main result on the Albanese map of a geometrically-special normal projective variety (Theorem \ref{thm:main_result_I}).   
To prove Theorem \ref{thm:main_result_I}, we will make use of  the "orbifold base" (as defined below) of the Albanese fibration.  

\begin{definition}\label{def:orbifold_base} Let $f:X\to Y$ be a surjective morphism of normal projective varieties over $k$ with   connected fibres. We define the orbifold divisor $\Delta(f)$ on $Y$ as follows.  Let $D\subset Y$ be a prime divisor of $Y$. Let $F_1,\ldots, F_r$ be the irreducible divisors of $X$ which map surjectively to $D$ via $f$. Then, we  may write the scheme-theoretic fibre of $f$ over $D$ as $f^\ast D = R+\sum_{k} t_k\cdot F_k $ with $R$ an $f$-exceptional divisor of $X$ with $f(R)\subsetneq D$.  For each irreducible Weil divisor  $D\subset Y$, we define   $$m_f(D) := \mathrm{inf}\{t_k\}.$$  We   define $\Delta := \sum_{D} \left(1-\frac{1}{m_f(D)}\right)  D$, where the sum runs over all prime divisors of $Y$.
\end{definition}

\begin{proof}[Proof of Theorem \ref{thm:main_result_I}]
We first reduce to $k=\mathbb{C}$ (so that we can appeal to Yamanoi's work \cite{Yamanoi1}). To do so, choose  a countable algebraically closed subfield $k_0\subset k$   such that there is a normal projective variety $X_0$ over $k_0$ with $X_{0,k}\cong X$ such that $X_0$ is geometrically special over $k_0$ (Lemma \ref{lem:descent}). Since $k_0$ is countable, we may choose an embedding $k_0\to \mathbb{C}$. Since $X_0$ is geometrically-special over $k_0$, we have that $X_{0,\mathbb{C}}$ is geometrically-special over $\mathbb{C}$. By standard properties of the Albanese map, it suffices to prove the theorem for $X_{0,\mathbb{C}}$. Thus, to prove the theorem, we may and do assume that $k=\mathbb{C}$.

 Let $X$ be a normal proper  geometrically-special variety over $k$, and let $\mathrm{alb}:X\to \mathrm{Alb}(X)$ be the Albanese map. By Corollary \ref{cor:geom_special_albanese}, the morphism $\mathrm{alb}$ is surjective with connected fibres. Therefore, we may consider   the associated   orbifold divisor $\Delta(\mathrm{alb})\subset \mathrm{Alb}(X)$ of $\mathrm{alb}$ (Definition \ref{def:orbifold_base}). To prove the theorem,  it suffices to show that $\mathrm{alb}$ has no multiple fibers in codimension one (Definition \ref{defn:mult}).   
 
 We argue by contradiction and suppose that $\mathrm{alb}$ has a multiple fiber in codimension one. In particular, the orbifold divisor $\Delta(\mathrm{alb})$ is non-empty.
 
 Let $\Delta'$ be an irreducible component of the support of $\Delta(\mathrm{alb})$. Then,  by Ueno's fibration theorem  \cite[Chapter~IV, Theorem~10.9]{Ueno},  there is a non-zero abelian variety $A$, a homomorphism $F:\mathrm{Alb}(X)\to A$, and a  big   divisor $D$ on $A$ such that $ \Delta' = F^{-1}(\mathrm{Supp} \ D)$. (Explicitly: $D$ is defined to be the general type base of the Ueno fibration for $\Delta'$. The fact that $D$ is a \emph{big} divisor follows from the fact that $D$ is of general type, so that it is  in particular a divisor with finite stabilizer. Such divisors are ample \cite[Lemma~8.5.6]{BombieriGubler}).
 
  We now consider the composed morphism $f:X\to \mathrm{Alb}(X)\to A$ and the orbifold divisor $\Delta = (1-\frac{1}{m})D$, where $m$ is the multiplicity of $D$ in the orbifold divisor $\Delta(f)$ of $f$.  If $m=1$, then the fibre over $D$ has a component of multiplicity one. In particular, if $m=1$, the fibre of $\mathrm{alb}$ over $\Delta'$ also has a component of multiplicity one. (As the referee pointed out to us, this component might be    $\mathrm{alb}$-exceptional a priori, so that $\Delta(\mathrm{alb})$ might still be non-trivial; see Remark \ref{remark:exc}.)  Thus, to prove the theorem, it suffices to show that $m=1$. We argue by contradiction and assume that $m>1$.
   
   Note that $f$ is surjective with connected fibres (and that $X\to (A,\Delta)$ might not be orbifold). Write $f^*D = R+ \sum_{k=1}^r t_k F_k$, where $F_1,\ldots, F_r$ are the irreducible divisors in $f^{-1}(D)$ mapping surjectively onto $D$ and $R$ is an $f$-exceptional divisor of $X$ with $f(R)\subsetneq D$. Let $Z$ be the image of $R$ in $A$ and note that $Z$ is of codimension at least two (and possibly empty) in $A$.
 
 Since $X$ is geometrically-special over $k$, we may (and do)  choose a covering set $g_i:(C,c)\to (X,x)$ with  $x\not\in f^{-1} (D)$ and $C$ a smooth projective curve over $k$. Let $\pi_C:C\to \mathbb{P}^1_k$ be a finite   morphism of degree at most $\mathrm{genus}(C)+1$ and  define $C^0  = \pi_C^{-1}(\mathbb{A}^1_k)$. Consider the covering set $f_i:(C,c)\to (A,a)$, where $f_i := f\circ g_i$ and $a:= f(x)$. Note that $f_i(C) \not\subset D$ and that    $f_i$ is    "orbifold modulo $Z$"  in the sense that the desired inequality on multiplicities     holds outside of $Z$. More precisely, we have
 \begin{eqnarray}\label{inequality}
 \# f_i^{-1}(D) &=& \# f_i^{-1}(D \setminus Z) + \# f_i^{-1}(Z) \leq \frac{1}{m} \deg f_i^\ast D + \# f_i^{-1}(Z).
 \end{eqnarray}
  Let $L$ be an ample line bundle on $A$ and let $\epsilon >0$ be  a rational  number such that $\Delta -2\epsilon L$ is big. 
Applying \cite[Proposition~5]{Yamanoi1} (with "$S$" in \emph{loc. cit.} a point), we obtain that  there is a real number $\kappa = \kappa(C, A,Z,L,\epsilon)>0$ depending only on  $C, A, Z,	L,$  $ \epsilon$ and a finite set $P = P(A,Z, L,\epsilon)$ of positive-dimensional abelian subvarieties of $A$ depending only on $A$, $Z$, $L$,  $\epsilon$ such that either 
  \begin{eqnarray}
  \# f_i^{-1}(Z) &\leq & \kappa   + \epsilon \cdot \deg f_i^\ast L
  \end{eqnarray}   or $f_i(C)$ is contained in the subset   $\Phi:=  \cup_{B\in P} \Phi(B,Z)\subset A$, where 
  \begin{eqnarray}
  \Phi(B,Z) := \{x\in A \ | \ \dim((x+B) \cap Z)\geq \dim B-1\}.
  \end{eqnarray} Note that $\Phi(A,Z) =\emptyset$ and that $\Phi$ is a proper closed subset of $A$. 
  
    Since $(f_i)_i$ is a covering set,   replacing   $(f_i)_i$    by a subsequence if necessary, we may and do assume that, for each $i$, we have that $f_i(C)\not\subset \Phi$. In particular, for each $i$, we have 
  \begin{eqnarray}\label{eqn2}
  \# f_i^{-1}(Z) &\leq & \kappa   + \epsilon\cdot  \deg f_i^\ast L. 
  \end{eqnarray}

  Now, as in the proof of Theorem \ref{thm:yamanoi_erwan}, it follows from \cite[Corollary~2]{Yamanoi1}  that there is a real number $\kappa_2$ depending only on $C, A, D, L, \epsilon$  such that
  \begin{eqnarray}\label{eqn3}
 \deg f_i^\ast D &\leq& \# f_i^{-1}(D) + \kappa_2  + \epsilon \cdot \deg f_i^\ast L
  \end{eqnarray}

Substituting inequalities (\ref{inequality}) and (\ref{eqn2}) in (\ref{eqn3}), we see that    \begin{eqnarray}
  \deg f_i^\ast D &\leq  & \# f_i^{-1}(D) + \kappa_2  + \epsilon \cdot \deg f_i^\ast L \\ &\leq &\frac{1}{m} \deg f_i^\ast D +  \kappa   + \epsilon \cdot \deg f_i^\ast L +\kappa_2  + \epsilon \cdot \deg f_i^\ast L.
  \end{eqnarray} 
  
  Since $\Delta = \left(1-\frac{1}{m}\right) D$, we may rewrite this as 
\begin{eqnarray}
  \deg f_i^\ast (\Delta - 2\epsilon L)  \leq   \kappa   +\kappa_2.
  \end{eqnarray}   
  
  Note that, as $m>1$, the divisor $\Delta - 2\epsilon L$ is big.
  Therefore, as in  Remark  \ref{remark:independence_of_choice_of_ample}, since $\Delta-2\epsilon L$ is big, we may write $\Delta -2\epsilon L = M  + F$, where $M$ is ample and $F$ is effective.  Replacing the covering set $(f_i)_i$ by a subsequence if necessary, we may and do  assume that $f_i(C)\not \subset F$ for every $i$. Then, for every $i$, we have that 
  \begin{eqnarray}
  \deg f_i^\ast M \leq  \deg f_i^\ast M + \deg f_i^\ast F =\deg f_i^\ast (\Delta - 2\epsilon L) \leq   \kappa + \kappa_2.
  \end{eqnarray}  
  Since $\kappa + \kappa_2$ is a real number depending only on $C$, $A$, $D$, $L$,  $\epsilon$, we conclude that  the moduli space of pointed maps $(C,c)\to (A,a)$ is positive-dimensional (as it has a component containing an infinite subsequence of $(f_i)_i$). Therefore, by Mori's bend-and-break (see \cite[Proposition~3.1]{DebarreBook}), there is a rational curve in $A$ (passing throug $a$), contradicting the fact that $A$ has no rational curves. This  contradiction  concludes the proof.
 \end{proof}
 
\begin{remark}\label{remark:referee2}
In an earlier version of our paper, we used that the map $\mathrm{alb}: X\to (\mathrm{Alb}(X),\Delta(\mathrm{alb}))$ is orbifold in the proof of Theorem \ref{thm:main_result_I}. This is however \emph{a priori} false, as a referee pointed out to us.  The issue comes from  $\mathrm{alb}$-exceptional components over $\Delta$.   We deal with this "codimension two" issue by using \cite[Proposition~5]{Yamanoi1} in our proof.
\end{remark} 
 
 \begin{remark}[Analytic analogue]\label{remark:analytic}
 As already mentioned in the introduction, the analytic analogue of Theorem \ref{thm:main_result_I} was known before. Indeed, as stated in  \cite[Proposition~5.1]{CampanaClaudonStability}, for $X$ a Brody-special normal projective variety, the Albanese map is surjective, has connected fibres, and has  no multiple fibres in codimension one; this can be deduced directly from     \cite[Main~Theorem]{NWY1}, \cite[Corollary~1.4]{LuWinkelmann}, and \cite[Theorem~13]{KawamataChar}.   The codimension two phenomenon discussed in Remark \ref{remark:referee2} also occurs in the proof of the analytic analogue and is  dealt with accordingly in the proof of \cite[Theorem~4.15]{LuWinkelmann}. Note that in our proof of Theorem \ref{thm:main_result_I} we   can not appeal to \cite{LuWinkelmann}, as they only deal with holomorphic maps from $\mathbb{C}$ to $X$. This is why we resort to Yamanoi's results in the above proof. \end{remark}

  \begin{remark}\label{remark:exc} Given a surjective morphism with connected fibres $X\to Y$, 
  there is a subtle difference between "having no multiple fibres in codimension one" and having a trivial orbifold divisor. We stress that we do not show that the Albanese fibration $X\to \mathrm{Alb}(X) $ of a geometrically-special has trivial orbifold divisor (or is "base-special").  Indeed, our proof above only shows that the orbifold base of $f:X\to \mathrm{Alb}(X)\to A$ is trivial. Since  a multiplicity one component of the fibre of $f$ over a codimension one point might be $\mathrm{alb}$-exceptional,  one can not conclude that the Albanese fibration has trivial orbifold divisor. This subtle issue seems to have been overlooked in    \cite[Proposition~5.3]{Ca04},    \cite[Proposition~5.1]{CampanaClaudonStability} and    \cite{LuWinkelmann}.
  \end{remark}
 
\section{Virtually abelian fundamental groups}\label{section:pi1}
In this section we show that the image of the fundamental group of a geometrically-special smooth projective variety over $\CC$ along a linear representation is virtually abelian (Theorem \ref{thm:pi1_geomhyp}). 
Our proof  follows closely Yamanoi's strategy \cite{Yam10}, and we will indicate as carefully as possible how to adapt Yamanoi's line of reasoning to prove Theorem \ref{thm:pi1_geomhyp}. 

The key observation is that the formal properties of Brody-special varieties used by Yamanoi also hold for geometrically-special varieties (e.g., the Albanese map of a Brody-special proper variety is surjective, much like the Albanese map of a geometrically-special proper variety). The novel technical result is arguably Theorem \ref{prop:3.1a}. 

\subsection{The case of $p$-unbounded representations}  \label{section:31}

It is reasonable to state that   \cite[Proposition~3.1]{Yam10} is one of the crucial results of Yamanoi's  paper on fundamental groups of Brody-special  varieties. We start with establishing the geometrically-special analogue of his result; see Proposition \ref{prop:3.1} for a precise statement.

Recall that, if $L$ is a field, a group scheme $G$ over $L$ is an \emph{ almost simple algebraic group over  $L$} if $G$ is a non-commutative  smooth affine   (geometrically) connected group scheme  over $L$ and $G$ has no infinite proper closed  normal subgroups.

 If $p$  is a prime number   and $G$ is  an almost simple algebraic group over a finite field extension $K$ of $\mathbb{Q}_p$, then  we say that a representation $\rho: \pi_1(X)\to G(K)$ is \emph{$p$-bounded} if its image $\rho(\pi_1(X))$
 is contained in a maximal compact subgroup of $G(K)$. If $\rho$ is not $p$-bounded, then we say that $\rho$ is \emph{$p$-unbounded}. 
 
 In what follows we will prove that the fundamental group of a smooth projective geometrically-special variety does not admit a Zariski-dense representation into an almost simple algebraic group. To do so, we will argue case-by-case.  The crucial (arguably most technical) case is that of a \emph{big} $p$-unbounded   representation considered first.  (For the reader's convenience, we recall that   a representation $\rho:\pi_1(X)\to \mathrm{GL}_n(\CC)$ is   \emph{big} if there is a countable  collection of proper closed subsets $X_i\subsetneq X$ with $i=1,\ldots$ such that, for every positive-dimensional subvariety $Z\subset X$ containing a point of $X\setminus \cup_{i=1}^\infty X_i$, the group $\rho[\mathrm{Im}(\pi_1(\tilde{Z})\to \pi_1(X))]$ is infinite, where   $\tilde{Z}\to Z$ is a desingularization of $Z$.)

\begin{theorem}\label{prop:3.1a}  
Let $X$ be a smooth projective variety over $\CC$, and let $K$ be a finite extension of $\QQ_p$. Let $G$ be an almost simple  algebraic group   over $K$. Assume that there exists a \textbf{big $p$-unbounded} representation $\rho:\pi_1(X)\to G(K)$ whose image is Zariski-dense in $G$. Then $X$ is not geometrically-special over $\CC$.  
\end{theorem}
\begin{proof}
As in Yamanoi's proof of \cite[Proposition~3.1]{Yam10} we will use Zuo's results on spectral coverings \cite{ZuoChern, Zuo99} (see also \cite{CCEZuo}). 
For instance, by Zuo's work (\cite[Theorem~1]{ZuoChern} and \cite[Theorem~2]{ZuoChern}),     as $\rho$ is big, the variety $X$ is ``Chern-hyperbolic''.  In particular, there is a proper closed subset $\Delta\subsetneq X$ such that every non-constant morphism $f:\mathbb{P}^1_{\CC}\to X$ factors over $\Delta$, i.e., we have $f(\mathbb{P}^1_{\CC}) \subset \Delta$. Moreover, the variety $X$ is of general type\footnote{At this point, it is clear that $X$ is not special. However, proving that $X$ is not geometrically-special requires many more arguments.}.

We will  make explicit use  of Zuo's spectral covering  $\pi:X^s\to X$; see \cite[p~554]{Yam10} for an explanation of the notation.
As is explained in \emph{loc. cit.}, the variety  $X^s$ is equipped with holomorphic one-forms $\omega_1, \dots, \omega_{\ell}$ that are contained in $\mathrm{H}^0(X^s, \pi^\ast\Omega^1_X)$ and such that $\pi: X^s \to X$ is unramified outside $\bigcup_{\omega_i\neq \omega_j}(\omega_i-\omega_j)_0.$ Moreover, the forms $\omega_1,\ldots,\omega_{\ell}$ are used to construct the adapted Albanese map  $\alpha:X^s\to A$. We  note that thanks to \cite{ZuoChern}, the morphism $\alpha$ is generically finite onto its image, so that $X^s$ has maximal Albanese dimension (because the adapted Albanese map factorizes through the Albanese map). Moreover,   since $X$ is of general type, it follows that $X^s$ is of general type.  Therefore, it follows from Yamanoi's theorem (Theorem \ref{thm:yamanoi}) that $X^s$ is pseudo-algebraically hyperbolic. In particular, replacing $\Delta\subset X$ by a suitable larger proper closed subset, we may  assume that $X^s$ is algebraically hyperbolic modulo $\pi^\ast \Delta$. Now, we fix  an ample line bundle $M$ on $X$, so that (by the algebraic hyperbolicity of $X^s$ modulo $\pi^\ast \Delta$)  there is a real number $\gamma$ depending only on $X$, $M$, and $\pi^\ast \Delta$ such that, for every smooth projective connected curve $W$ over $\CC$ and every morphism $g:W\to X^s$ with $g(W)\not\subset \pi^\ast \Delta$, the inequality 
 \begin{eqnarray}\label{yamanoi_ineq}
 \deg g^\ast (\pi^\ast M) &\leq& \gamma \cdot \mathrm{genus}(W)
 \end{eqnarray} holds.  
 
To prove the proposition,  we argue by contradiction.  Thus, assume that $X$ is geometrically-special over $\CC$.    Define $U:=X\setminus \Delta$, choose a smooth projective connected curve $C$ of genus at least two, a point $c_0$ in $C$, a point $x$ in $ U$, and a covering set of non-constant morphisms  $f_k:(C,c_0)\to (X,x)$ with $k=1,\ldots$.   By the definition of a covering set, replacing the sequence of morphisms $(f_k)_{k=1}^\infty$  by a suitable subsequence if necessary, we have that, for every proper closed subset $Z\subset X$, there are only finitely many integers $k$ such that $f_k(C)\subset Z$.  
 
For every $k=1,\ldots$, we consider the pull-back $V_k\to X^s$ of $f_k:C\to X$ along $\pi:X^s\to X$.   Let $W_k$ be the normalization of $V_k$, and let $g_k:W_k\to X^s$ denote the composed morphism $W_k\to V_k\to X^s$. (Note that $W_k$ is a, possibly disconnected,    smooth projective   curve over $\CC$.)   This gives us, for every $k=1,\ldots,$ the following commutative diagram
\begin{equation}
  \label{eq:lifting}
  \vcenter{\xymatrix{
      W_k
      \ar[d]^{\rho_k}
      \ar[r]^{g_k}
      &
      X^s
      \ar[d]^{\pi} \ar[r]^{\alpha} & A
      \\
      C
      \ar[r]^{f_k}
      &
      X
  }}
  \end{equation}
 Now, since $f_k(c_0) = x\not\in \Delta$, for every connected component $W_k'$ of $W_k$, we have that  $g_k(W_k')\not\subset 
 \pi^\ast \Delta$. In particular, using that $X^s$ is algebraically hyperbolic modulo $\pi^\ast \Delta$ and letting $\mathrm{genus}(W_k)$ denote $\sum \mathrm{genus}(W_k')$ where the sum runs over all connected components $W_k'$ of $W_k$,  we see that the inequality  
\begin{eqnarray}\label{ineq3}
\deg \pi \cdot \deg f_k^\ast M &=& \deg (f_k \circ \rho_k)^\ast M = \deg g_k^\ast(\pi^\ast M) \leq \gamma \cdot \mathrm{genus}(W_k)
\end{eqnarray} holds; see (\ref{yamanoi_ineq}) above. Note that we used that $\deg \rho_k=\deg \pi$.

Let $D\subset X^s$ be the (support of the) ramification divisor of $\pi: X^s \to X$. Then $D \subset \bigcup_{1\leq i<j\leq \ell} \Xi_{ij},$ where   $\Xi_{ij}:=(\omega_i-\omega_j)_0.$  Let $\omega_{ij}:=\omega_i-\omega_j.$ The idea is to use the forms $\omega_i$ to ``control'' the genus of $W_k$ in terms of the degree of $g_k$. (Since there is no a priori control on the branch points of  $W_k\to C$, the genus of $W_k$ may tend to infinity as $k$ tends to infinity.)
 
 To do so, 
  by the Riemann-Hurwitz formula $K_{W_k}=\rho_k^*(K_C)+R_k$, it suffices to ``control'' the degree of the ramification divisor $R_k$ of $\rho_k$. Note that  
\begin{eqnarray}\label{ineq2} 
\deg R_k &\leq& (\deg \rho_k) (\# g_k^{-1}(D))\leq (\deg \rho_k)\sum_{1\leq i<j\leq \ell} \# g_k^{-1}(\Xi_{ij}).
\end{eqnarray} In particular, 
\begin{eqnarray}\label{ineq20} 
\mathrm{genus}(W_k) &\leq& ( \deg \rho_k)\cdot \mathrm{genus}(C)+ (\deg \rho_k)\sum_{1\leq i<j\leq \ell} \# g_k^{-1}(\Xi_{ij}).
\end{eqnarray}
Let $1\leq i < j \leq \ell$ be such that   $g_k^\ast\omega_{ij}\neq 0.$ 
Consider $\beta$ a holomorphic section of $K_C$ 
and the meromorphic function $\eta=g_k^{\ast}\omega_{ij}/\rho_k^{\ast}\beta.$ We can consider $g_k^\ast\omega_{ij}$ as a holomorphic section of $\rho_k^*\Omega_C$, so that  $\deg(\eta=\infty)= (\deg \rho_k) \deg K_C$.
Therefore,   it follows that
\begin{eqnarray}\label{ineq_points1}
\# g_k^{-1}(\Xi_{ij}) &\leq& \deg(\eta=0)=\deg(\eta=\infty) \\ &\leq& (\deg \pi) \deg K_C \leq 2 (\deg \pi) \mathrm{genus}(C). \nonumber
\end{eqnarray}
Now, let    $1\leq i < j \leq \ell$ be such that   $g_k^\ast\omega_{ij}= 0.$ Following Yamanoi (see \cite{Yam10} p.557 for details), we consider the Albanese map $b: X^s \to B$ with respect to $\omega_{ij}$. Also, we let  $S \to b(X^s)$ be the normalization of $b(X^s)$, so  that $b$ factors as 
$$
\xymatrix{ 
X^s \ar[r]^c & S \ar[r]^\psi & B.
}
$$  
Then, the image $c(\Xi_{ij})$ of $\Xi_{ij}$ is a point $P$ in $S$. Let $L$ be an ample line bundle on $S$. Then,  Yamanoi  constructs (see \cite{Yam10} p.558 for details), for every positive integer $n$, a positive integer $l_n$ and a divisor $D_n$ on $S$ such that the following three properties hold.
\begin{enumerate}
\item  The sequence of integers $(l_n)_{n=1}^\infty$ is strictly increasing and satisfies $\lim_{n\to \infty} \frac{l_n}{n} =0$;
\item  The divisor $D_n$ is contained in $|L^{\otimes l_n}|$;  
\item  If $c\circ {g_k}(y)=P$, then $\ord_y(c\circ {g_k})^*D_n \geq n.$
\end{enumerate}

 Now, since $\pi^\ast M$ is ample, we may choose an integer $N_0$ such that $N_0 \pi^\ast M-c^\ast L \geq 0$.
Next, we choose an integer $n$ such that  $\frac{l_n}{n} < \frac{1}{(\ell+1)^3} $; such an integer exists as $l_n=o(n)$ (see (1) above). Now, by our choice of covering set $(f_k)_{k=1}^\infty$, there is an integer $k_{i,j}$ such that, for all $k\geq k_{i,j}$, we have that $f_k(C)\not\subset \pi( c^{-1}(D_n))$. In particular, for all $k\geq k_{i,j}$, we have that $c\circ g_k(W_k)\not\subset D_n$. Therefore, for $k\geq k_{i,j}$, we have that  
\begin{eqnarray}\label{ineq5}
\# g_k^{-1}(\Xi_{ij}) &\leq& \# (c\circ g_k)^{-1}(P)\leq \frac{1}{n}\deg (c\circ g_k)^\ast D_n \leq \frac{l_n}{n}\deg (c\circ g_k)^\ast L \\
& \leq& \frac{1}{(\ell+1)^3} \deg (c\circ g_k)^\ast L  \leq   \frac{N_0}{(\ell+1)^3} \cdot \deg g_k^\ast \pi^\ast M  \nonumber
\end{eqnarray}

Define $$k_0 = \max(k_{i,j} \  | \ 1\leq i <j \leq \ell \textrm{ such that } g_k^\ast \omega_{ij} \neq 0\}.$$
By combining  (\ref{ineq5}) with (\ref{ineq_points1}) and (\ref{ineq20}), we obtain that, for all $k\geq k_0$, 
\begin{eqnarray*}\label{ineq_final}
\mathrm{genus}(W_k) &\leq& ( \deg \rho_k)\cdot \mathrm{genus}(C)+ (\deg \rho_k)\sum_{1\leq i<j\leq \ell} \# g_k^{-1}(\Xi_{ij}) \\
&\leq&  (\deg \rho_k)\cdot \Big{(} \mathrm{genus}(C)+   \sum_{\substack{1\leq i< j\leq \ell \nonumber \\ g_k^\ast\omega_{ij}\neq  0}}\left(   \# g_k^{-1}(\Xi_{ij})\right) + \  \sum_{\substack{1\leq i< j\leq \ell \\ g_k^\ast\omega_{ij}=  0.}}\left(   \# g_k^{-1}(\Xi_{ij})\right) \Big{)}   \nonumber \\ 
&\leq &    3 (\deg \pi)^2 \ell^2 \cdot \mathrm{genus}(C)+  \frac{\ell^2 N_0}{(\ell+1)^3}   \deg g_k^\ast \pi^\ast M  \nonumber \\ 
&=& 3 (\deg \pi)^2 \ell^2 \cdot \mathrm{genus}(C)+  \frac{N_0}{\ell+1}   \deg g_k^\ast \pi^\ast M. \nonumber
\end{eqnarray*}
By combining this upper bound for the genus of $W_k$   with   (\ref{ineq3}), we obtain that 
$$
\deg g_k^\ast(\pi^\ast M)\leq \gamma\cdot  \mathrm{genus}(W_k) \leq   3 \gamma  (\deg \pi)^2 \ell^2 \cdot \mathrm{genus}(C)+  \frac{N_0}{\ell+1}   \deg g_k^\ast \pi^\ast M
$$
 This implies that 
 $$
\deg g_k^\ast(\pi^\ast M)\leq   3 \gamma  N_0(\deg \pi)^2 (\ell+1)^3 \cdot \mathrm{genus}(C).
$$
In particular, we obtain that 

\begin{eqnarray}\label{ineq}
\deg f_k^\ast M &=& \frac{1}{\deg \pi} \deg g_k^\ast \pi^\ast M \leq  3 \gamma N_0  (\deg \pi) (\ell+1)^3 \cdot \mathrm{genus}(C).
\end{eqnarray} 
 Since the RHS of (\ref{ineq}) is independent of $k$, it follows that $(f_k:(C,c_0)\to (X,x))_{k=1}^\infty$ is an infinite sequence of pairwise distinct morphisms of bounded degree. In particular,  it follows from   Mori's bend-and-break (see \cite[Proposition~3.1]{DebarreBook}) that there is a rational curve in $X$ containing $x$. Since $x\not\in \Delta$, this  contradicts   the fact that all rational curves of $X$ are contained in $\Delta$.
 
We conclude that $X$ is not geometrically-special.
\end{proof}

\begin{proposition}[The analogue of 3.1 in \cite{Yam10}]\label{prop:3.1}
Let $X$ be a smooth projective connected variety over $\CC$, and let $K$ be a finite extension of $\QQ_p$. Let $G$ be an almost simple algebraic group   over $K$. Assume that there exists a  \textbf{$p$-unbounded} representation $\rho:\pi_1(X)\to G(K)$ whose image is Zariski-dense in $G$. Then $X$ is not geometrically-special over $\CC$.
\end{proposition}
\begin{proof}  By adapting the arguments in the first paragraph of Yamanoi's proof of \cite[Proposition~3.1]{Yam10}, we will reduce to the case that $\rho$ is big, so that the result follows from Theorem \ref{prop:3.1a}.   

Assume, for a contradiction that $X$ is geometrically-special over $\CC$.
Note that, if $X'$ is a smooth projective connected variety over $\CC$ which is birational to $X$, then $X'$ is geometrically-special (Lemma \ref{lem:surj}). Therefore, 
defining $H:=\ker \rho$, replacing $X$ by a  smooth projective  variety which is birational to $X$ if necessary, by the theory of Shafarevich maps (see \cite[p.~185]{Kollar} or \cite{CampanaShaf}), we have   the $H$-Shafarevich morphism
\[
\mathrm{sh}^H_X:X\to \mathrm{Sh}^H(X).
\]  Let $F$ be a general fiber of  $\mathrm{sh}^H_X$ and let $\pi_1(F)_X$ be the image of the natural morphism of groups $\pi_1(F)\to \pi_1(X)$. Then, by definition of the $H$-Shafarevich map, the image 
\[
\rho(\pi_1(F)_X)\subset G(K)
\] is finite. It now follows from \cite[Lemma~2.2.3]{Zuo99} that there is a smooth projective connected variety $Y$ and a morphism $Y\to X$ which is the composition of a proper birational surjective morphism and a finite \'etale morphism such that, if $Y\to \Sigma$ denotes the Stein factorization of the composed morphism $Y\to X\to \mathrm{Sh}^H(X)$, then there exists a representation 
\[
\rho_{\Sigma}: \pi_1(\Sigma)\to G(K)
\] such that $\pi_1(Y)\to G(K)$ factors over $\rho_{\Sigma}$.  By construction (see \cite[Proposition~2.2.2]{Zuo99}), the representation $\rho_{\Sigma}$ is big and has Zariski dense image in $G$. Let $\Sigma'$ be a resolution of singularities of $\Sigma$.  
Note  that the variety $Y$ is geometrically-special (Lemma \ref{lem:surj} and Lemma \ref{lem:cw_geom_special}). In particular, since $Y$ dominates $\Sigma$, it follows that $\Sigma$ and $\Sigma'$ are geometrically-special (Lemma \ref{lem:surj}). Thus, $\Sigma'$ is a smooth projective connected geometrically-special variety over $\CC$ whose fundamental group has a big $p$-unbounded representation $\pi_1(X)\to G(K)$. This contradicts Theorem \ref{prop:3.1a}.
\end{proof}

\subsection{Period maps} To prove that every linear quotient of the fundamental group of a geometrically-special variety is virtually abelian, we will use finiteness results for pointed maps to period domains. To state the necessary results, we recall the definition of a period map following Schmid  \cite[Section 3]{Schmid} (see also \cite[Section~4]{JLitt}). 

Let $H$ be a finitely-generated free $\mathbb{Z}$-module, $k$ an integer, and $\{h^{p,k-p}\}$ a collection of non-negative integers with $h^{p,k-p}=h^{k-p, p}$ for all $p$, $$\sum_p h^{p, k-p}=\text{rk}_{\mathbb{Z}} H.$$ Let $\hat{\mathscr{F}}$ be the flag variety parametrizing decreasing, exhaustive, separated filtrations of $H_\mathbb{C}$, $(F^\bullet)$, with $\dim F^p= \sum_{i\geq p} h^{i, k-i}$.
Let $\mathscr{F}\subset \hat{\mathscr{F}}$ be the analytic open subset of $\hat{\mathscr{F}}$ parametrizing those filtrations corresponding to $\mathbb{Z}$-Hodge structures of weight $k$, i.e.~those filtrations with $$H_\mathbb{C}=F^p+\overline{F^{k-p+1}}$$ for all $p$.
Now suppose $q$ is a non-degenerate bilinear form on $H_{\mathbb{Q}}$, symmetric if $k$ is even and skew-symmetric if $k$ is odd.  Let $D\subset \mathscr{F}$ be the locally closed analytic subset of $\mathscr{F}$ consisting of filtrations corresponding to polarized Hodge structures (relative to the polarization $q$), i.e.~the set of filtrations $(F^\bullet)$ in $\mathscr{F}$ with $$q_{\mathbb{C}}(F^p, F^{k-p+1})=0 \text{ for all $p$}$$ and $$q_{\mathbb{C}}(Cv, \bar v)>0$$ for all nonzero $v\in H_{\mathbb{C}}$, where $C$ is the linear operator defined by $C(v)=i^{p-q}v$ for $$v\in H^{p,q}:= F^p\cap \overline{F^q}.$$

Let $G=O(q)$ be the orthogonal group of $q$; it is a $\mathbb{Q}$-algebraic group. Write $G_{\mathbb{Z}}=G(\mathbb{Q})\cap \mathrm{GL}(H)$, and let $\Gamma\subset G_{\mathbb{Z}}$ be a finite index subgroup.

 If $X$ is smooth, then we say that a holomorphic map $X^{\an}\to \Gamma\backslash D$ is a \emph{period map on $X$} if it is locally liftable and horizontal (i.e., satisfies Griffiths transversality). More generally, for $X$  a possibly singular variety, a holomorphic map $X^{\an}\to \Gamma\backslash D$ is a \emph{period map on $X$ } if there is a desingularization $\tilde{X}\to X$ such that the composed morphism $\tilde{X}^{\an}\to X^{\an}\to \Gamma\backslash D$ is a period map on $\tilde{X}$.   
We will make use of the following finiteness property for varieties with a quasi-finite period map which is essentially    a reinterpretation of Deligne's finiteness result for monodromy representations.

\begin{lemma}[Deligne, Deligne-Griffiths-Schmid]\label{lem:deligne}
Let $X$ be a  variety over $\CC$. If $X$ admits a period map $X^{\an}\to \Gamma\backslash D$ with finite fibres, then $X$ is geometrically hyperbolic over $\CC$.
\end{lemma}
\begin{proof} 
This follows from Deligne's finiteness theorem for monodromy representations \cite{DeligneMonodromy} and the Theorem of the Fixed Part   \cite[7.24]{Schmid}. 
We refer the reader to  \cite[Theorem~5.1]{JLitt} for a detailed proof.
\end{proof}

\subsection{Period domains are far from being special}

In \cite[Section~4]{Yam10}, Yamanoi divides the task of proving his theorem into the rigid case and non-rigid case. To deal with the rigid case, he crucially uses that varieties which admit a non-constant period map are not Brody-special (i.e., do not admit a dense entire curve). The   analogue for geometrically-special varieties of the statement he uses    reads as follows.

 \begin{theorem}\label{thm:period_maps}
 Let $X$ be a variety over $\mathbb{C}$. If $X$ admits a non-constant period map, then $X$ is not geometrically-special over $\CC$.
 \end{theorem}
 \begin{proof} Let $p:X^{\an}\to \Gamma\backslash D$ be a non-constant period map.
  By Bakker-Brunebarbe-Tsimerman's resolution of Griffiths's conjecture \cite{BakkerBrunebarbeTsimerman}, the analytic closure of the image  of $p$ is the analytification $Y^{\an}$ of a quasi-projective variety $Y$ over $\CC$, and the induced dominant morphism $X^{\an}\to Y^{\an}$ is the analytification of a (non-constant dominant) morphism $f:X\to Y$. Note that $\dim Y >0$. Since $Y$ admits a quasi-finite (even injective) period map, it is geometrically hyperbolic (Lemma \ref{lem:deligne}). Since $X$ dominates the positive-dimensional geometrically hyperbolic variety $Y$, by Corollary \ref{cor:image_is_not_psgeom}, we conclude that $X$ is not geometrically-special  over $\CC$.
  \end{proof}

\subsection{The rigid case} We now deal with the case    of  a rigid representation $\rho:\pi_1(X) \to G(\CC)$.

 \begin{proposition}[The analogue of 4.1 in \cite{Yam10}]\label{prop:rigid}
 Let $X$ be a smooth projective connected variety over $\CC$, and let $G$ be an almost simple algebraic group over $\CC$. Assume that there exists a \textbf{rigid} representation $\rho:\pi_1(X)\to G(\CC)$ whose image $\rho(\pi_1(X))$ is Zariski-dense in $G$. Then $X$ is not geometrically-special over $\CC$.
 \end{proposition}
 \begin{proof}
Since $\rho$ is rigid, it can be defined over some number field $K$. By abuse of notation, we let $\rho:\pi_1(X)\to G(K)$ denote a model over $K$. \  For $p$ a  finite place of $K$, we consider the representation 
\[
\rho_p:\pi_1(X) \to G(K_p).
\] If there exists a finite place $p$ of $K$ such that $\rho_p$ is $p$-unbounded, then Proposition \ref{prop:3.1} implies that $X$ is not geometrically-special. Thus, we may and do assume that, for every finite place $p$ of $K$, the representation $\rho_p$ is $p$-bounded. In this case, as is explained by Yamanoi in the proof of \cite[4.1]{Yam10}, the subgroup $\rho^{-1}(G(\mathcal{O}_K))$ is of finite index in $\pi_1(X)$. Therefore, there is a finite \'etale cover $Y\to X$ such that $\rho(\pi_1(Y)) \subset G(\mathcal{O}_K)$.  Then, as $\rho|_{\pi_1(Y)}:\pi_1(Y)\to G(\mathcal{O}_K)$ is a rigid representation,  by a result of Simpson (see \cite[p.~58]{Simpson}), the variety $Y$ admits a   period map, say $p:Y^{\an}\to \Gamma\backslash D$, whose rational monodromy representation contains $\rho|_{\pi_1(Y)}$ as a direct factor. Since $\rho$ has Zariski dense image in $G$, it follows that the period map $p$ is non-constant. We conclude that $Y$ is not geometrically-special from Theorem \ref{thm:period_maps}.  In particular, as $Y\to X$ is finite \'etale, the variety $X$ is not geometrically-special (Lemma \ref{lem:cw_geom_special}). This concludes the proof.
 \end{proof}

\subsection{The general case}\label{section:43} To deal with the non-rigid case, we begin with the following result proven by Yamanoi which we  state in more generality than necessary.

\begin{proposition}\label{prop:yamanoi_character_variety}   
Let $X$ be a variety over $\CC$, and let $G$ be an almost simple algebraic group over $\CC$. Assume that there exists a non-rigid representation $\rho:\pi_1(X)\to G(\CC)$ whose image $\rho(\pi_1(X))$ is Zariski-dense in $G$. Then, for every prime number $p$, there exists a finite extension $K/\QQ_p$ and  a $p$-unbounded representation $\widetilde{\rho}:\pi_1(X)\to G(K)$ whose image is Zariski-dense in $G$.
\end{proposition}
\begin{proof} 
This is shown in the proof of \cite[Lemma~4.1]{Yam10}. (The argument only uses properties of character varieties.)
\end{proof}

 Proposition \ref{prop:yamanoi_character_variety}  allows us to deal with non-rigid representations by appealing to our result for $p$-unbounded representations.

 \begin{corollary}[The analogue of 2.1 in \cite{Yam10}]\label{cor:21}
 Let $X$ be a smooth projective connected variety over $\CC$, and let $G$ be an almost simple algebraic group over $\CC$. Assume that there exists a representation $\rho:\pi_1(X)\to G(\CC)$ whose image $\rho(\pi_1(X))$ is Zariski-dense in $G$. Then, the variety $X$ is not geometrically-special over $\CC$.
 \end{corollary}
 \begin{proof}
If  $\rho$ is rigid, then the result follows from  Proposition \ref{prop:rigid}, so   we may   assume that $\rho$ is non-rigid.  Let $p$ be a prime number. Since $\rho$ is a non-rigid representation whose image is Zariski-dense in $G$, it follows from Proposition \ref{prop:yamanoi_character_variety} that there is a $p$-adic field $K$ and a $p$-unbounded representation $\widetilde{\rho}:\pi_1(X)\to G(K)$ whose image is Zariski-dense in $G$. Then, by   Proposition \ref{prop:3.1}, the variety $X$ is not geometrically-special over $\CC$.  
 \end{proof} 
 
 \subsection{Linear quotients of fundamental groups}
 
Corollary \ref{cor:21} says that the fundamental group $\pi_1(X)$ of a smooth projective connected geometrically-special variety $X$ over $\CC$ does not admit a Zariski dense representation into an almost simple algebraic group. We now combine this result with the structure result for Albanese maps of such varieties to  show that linear quotients of $\pi_1(X)$ are virtually abelian.
 
 \begin{proof}[Proof of Theorem \ref{thm:pi1_geomhyp}] 
 
 Assume that $X$ is geometrically-special over $\CC$, and let $\rho:\pi_1(X)\to \mathrm{GL}_n(\CC)$ be a representation.   We follow Yamanoi's ``Proposition 2.1 implies Theorem 1.1'' in \cite{Yam10}.  Thus,  let $H$ be the Zariski closure of the image of $\rho$ in $\mathrm{GL}_n(\CC)$.  Since every finite \'etale cover of $X$ is geometrically-special (Lemma \ref{lem:cw_geom_special}), replacing $X$ by a finite \'etale cover, we may and do assume that $H$ is connected. Then, as $H/R(H)$ is an almost direct product of almost simple  algebraic groups, it follows from Corollary \ref{cor:21} that $H$ equals its radical $R(H)$.    	Therefore, the image $\rho(\pi_1(X))$ is a solvable group.
 
 We now apply a theorem of Campana  \cite[Theoreme~2.9]{CampanaEnsemble}. First, note that  the Albanese map of every finite \'etale covering $X'$ of $X$ is surjective (Corollary  \ref{cor:covers_have_surj_alb}). Therefore, it follows  from \cite[Theoreme~2.9]{CampanaEnsemble}   that there exists a finite \'etale morphism $Y\to X$ such that $\pi_1(Y)\to H$ factors over the induced group homomorphism $\pi_1(Y)\to \pi_1(\mathrm{Alb}(Y))$. Since $\pi_1(\mathrm{Alb}(Y))$ is abelian, this implies that $\rho(\pi_1(Y))$ is abelian. We conclude that $\rho(\pi_1(X))$ is virtually abelian.
 \end{proof}

  \begin{proof}[Proof of Theorem \ref{thm:pi1_geomhyp2}]
We follow and adapt Yamanoi's proof of \cite[Corollary~1.3]{Yam10}.

Let $X$ be a geometrically-special smooth projective variety  over $\CC$ such that there is  a big representation
$\rho:\pi_1(X)\to \mathrm{GL}_n(\CC)$.  Then, by Theorem \ref{thm:pi1_geomhyp}, the image of $\rho$ is virtually abelian.  Thus, as every finite \'etale cover of $X$ is geometrically-special (Lemma \ref{lem:cw_geom_special}), replacing $X$ by a suitable finite \'etale cover if necessary, we   may and do assume that the image of  $\pi_1(X)$ is a free abelian group. Since the image of $\rho$ is  a   torsion-free abelian group and $\pi_1(\mathrm{Alb}(X))$ is the torsion-free abelianization of $\pi_1(X)$, it follows that  $\rho$ factors over the homomorphism $\pi_(X)\to \pi_1(\mathrm{Alb}(X))$ induced by the   Albanese map $a_X:X\to \mathrm{Alb}(X)$.  To conclude the proof,   we now show that the Albanese map  $a_X:X\to \mathrm{Alb}(X)$ is birational. 

Since $X$ is geometrically-special, the Albanese map
$a_X$ is surjective with connected fibers (Corollary \ref{cor:geom_special_albanese}). Let
$F$ be a general fiber of $a_X$. Then
$
\rho\left( \mathrm{Im} (\pi_1(F) \to\pi_1(X)) \right)
$ is trivial. In particular, since $\rho$ is big, the fiber $F$ is a point. From this it follows that $a_X$ is birational, as required. 
 \end{proof}

\bibliographystyle{alpha}
\bibliography{orbi}{}
 
\end{document}